\documentstyle[twoside,amsmath, amssymb, mathrsfs, amsfonts, eucal, epsfig, 11pt]{article}

 \newtheorem{theorem}{Theorem}[section]
 \newtheorem{lemma}[theorem]{Lemma}
 \newtheorem{proposition}[theorem]{Proposition}
 \newtheorem{corollary}[theorem]{Corollary}
 
 \newenvironment{proof}{\begin{trivlist} \item[]{\em Proof.}}{\end{trivlist}}
 \newcount\refno
\def\CC{{\mathbb C}}
\def\DD{{\mathbb D}}

 \def\RR{{\mathbb R}}
 \def\NN{{\mathbb N}}

 \title{\bf Some aspects of the Bergman and Hardy spaces associated with a class of generalized analytic functions
 \thanks{{Supported by the National Natural
 Science Foundation of China (No. 12071295).}
 \newline
 \indent \,\,$^\dag$Corresponding author.
 \newline
 \indent\,\, E-mail: lizk@shnu.edu.cn (Zh.-K. Li); hhwei@cslg.edu.cn (H.-H. Wei).}}

\author{Zhongkai Li$^{1}$ and Haihua Wei$^{\dag,2}$\\
{\small $^{1}$Department of Mathematics, Shanghai Normal University}\\
{\small Shanghai 200234, China} \\
{\small $^{2}$School of Mathematics and Statistics, Changshu Institute of Technology}\\
{\small Changshu 215500, Jiangsu, China}
}

\date{}

\begin{document}
 \maketitle \setcounter{page}{1} \pagestyle{myheadings}
 \markboth{Li and Wei}{Some aspects of Bergman and Hardy spaces}

 \begin{abstract}
 \noindent
For $\lambda\ge0$, a $C^2$ function $f$ defined on the unit disk ${{\mathbb D}}$ is said to be
$\lambda$-analytic if $D_{\bar{z}}f=0$, where $D_{\bar{z}}$ is the (complex) Dunkl operator given by  $D_{\bar{z}}f=\partial_{\bar{z}}f-\lambda(f(z)-f(\bar{z}))/(z-\bar{z})$.
The aim of the paper is to study several problems on the associated Bergman spaces $A^{p}_{\lambda}({{\mathbb D}})$ and Hardy spaces $H_{\lambda}^p({{\mathbb D}})$ for $p\ge2\lambda/(2\lambda+1)$, such as boundedness of the Bergman projection, growth of functions, density, completeness, and the dual spaces of $A^{p}_{\lambda}({{\mathbb D}})$ and $H_{\lambda}^p({{\mathbb D}})$, and characterization and interpolation of $A^{p}_{\lambda}({{\mathbb D}})$.

 \vskip .2in
 \noindent
 {\bf 2020 MS Classification:} 30H20, 30H10 (Primary), 30G30, 42A45 (Secondary)
 \end{abstract}

 %%%%%%%%%%%%%%%%%%%%%%%%%%%%%%%%%%%%%%%%%%%%%%%%%%%%%%%%%%%%%%%%%%%%%%%%%%%%%%%%%%%%%%%%%%
\setcounter{page}{1}
%%%%%%%%%%%%%%%%%%%%%%%%%%%%%%%%%%%%%%%%%%%%%%%%%%%%%%%%%%%%%%%%%%%%%%%%%%%%%%%%%%%%%%%%%%

\section{Introduction}

For $\lambda\ge0$, the (complex) Dunkl operators $D_{z}$ and $D_{\bar{z}}$ on the complex plane $\CC$,
as substitutes of
$\partial_z$ and
$\partial_{\bar{z}}$, are defined by
\begin{align*}
    D_{z}f(z)&=\partial_z f+
    \lambda \frac{f(z)-f(\bar{z})}{z-\bar{z}},\\
    D_{\bar{z}}f(z)&=\partial_{\bar{z}}f-
      \lambda\frac{f(z)-f(\bar{z})}{z-\bar{z}}.
\end{align*}
The associated Laplacian, called the $\lambda$-Laplacian, is given by
$\Delta_{\lambda}=4D_{z}D_{\bar{z}}=4D_{\bar{z}}D_{z}$,
which can be written explicitly as
\begin{eqnarray*}
\Delta_{\lambda} f=\frac{\partial^2 f}{ \partial x^2}+\frac
{\partial^2 f}{\partial y^2} +\frac{2\lambda}{y}\frac{\partial
f}{\partial y}-\frac{\lambda}{y^2} [f(z)-f(\bar{z})],\qquad z=x+iy.
\end{eqnarray*}
A $C^2$ function $f$ defined on the unit disk $\DD$ is said to be
$\lambda$-analytic, if $D_{\bar{z}}f=0$; and $f$ is said to be $\lambda-$harmonic, if $\Delta_{\lambda}f=0$.

The measure on the unit disk $\DD$ associated with the operators $D_{z}$ and $D_{\bar{z}}$ is
\begin{eqnarray*}
d\sigma_{\lambda}(z)=c_{\lambda}|y|^{2\lambda}dxdy, \qquad z=x+iy,
\end{eqnarray*}
where $c_{\lambda}=\Gamma(\lambda+2)/\Gamma(\lambda+1/2)\Gamma(1/2)$ so that $\int_{\DD}d\sigma_{\lambda}(z)=1$.
For $0<p<\infty$, we denote by
$L^{p}(\DD;d\sigma_{\lambda})$, or simply by $L_{\lambda}^{p}(\DD)$,
the space of measurable functions $f$ on $\DD$ satisfying
$$
\|f\|_{L_{\lambda}^{p}(\DD)}:=\left(\int_{\DD}|f(z)|^{p}d\sigma_{\lambda}(z)\right)^{1/p}<\infty;
$$
and $L_{\lambda}^{\infty}(\DD)$, or simply $L^{\infty}(\DD)$, is the collection of all essentially
bounded measurable functions on $\DD$ with norm $\|f\|_{L^{\infty}(\DD)} ={\rm esssup}_{z\in\DD} |f(z)|$.
The associated Bergman space $A^{p}_{\lambda}(\DD)$, named the $\lambda$-Bergman
space, consists of those elements in $L_{\lambda}^{p}(\DD)$ that are
$\lambda$-analytic in $\DD$, and the norm of $f\in A^{p}_{\lambda}(\DD)$ is written as $\|f\|_{A_{\lambda}^{p}}$ instead of $\|f\|_{L_{\lambda}^{p}(\DD)}$.

The associated measure on the circle $\partial\DD\simeq[-\pi,\pi]$ is
\begin{eqnarray*}
dm_{\lambda}(\theta)=\tilde{c}_{\lambda}|\sin\theta|^{2\lambda}d\theta, \ \
\ \ \ \ \tilde{c}_{\lambda}=c_{\lambda}/(2\lambda+2).
\end{eqnarray*}
As usual, the $p$-means of a function $f$ defined on $\DD$, for $0<p<\infty$, are given by
\begin{eqnarray*}
M_p(f;r)=\left\{\int_{-\pi}^{\pi}|f(re^{i\theta})|^p\,dm_{\lambda}(\theta)\right\}^{1/p},\qquad 0\le r<1;
\end{eqnarray*}
and $M_{\infty}(f;r)=\sup_{\theta}|f(re^{i\theta})|$.
The $\lambda$-Hardy space $H_{\lambda}^p(\DD)$ is the collection of $\lambda$-analytic functions on $\DD$ satisfying
$$
\|f\|_{H_{\lambda}^p}:=\sup_{0\le r <1}M_p(f;r)<\infty.
$$
Obviously $H_{\lambda}^{\infty}(\DD)$ is identical with $A_{\lambda}^{\infty}(\DD)$. Note that for $0<p<1$, $\|\cdot\|_{X}$ with $X=A_{\lambda}^p(\DD)$ or $H_{\lambda}^p(\DD)$ is not a norm, however $\|f-g\|_X^p$ defines a metric.

It was proved in \cite{LL1} that $f$ is $\lambda$-analytic in $\DD$ if and only if $f$ has the series representation
\begin{eqnarray}\label{anal-series-1-1}
f(z)=\sum_{n=0}^{\infty}c_{n}\phi_{n}^{\lambda}(z), \qquad |z|<1,
\end{eqnarray}
where
\begin{eqnarray*}
\phi_{n}^{\lambda}(z)=\epsilon_n\sum_{j=0}^{n}\frac{(\lambda)_{j}(\lambda+1)_{n-j}}
{j!(n-j)!}\bar{z}^{j}z^{n-j}, \qquad n\in\NN_0,
\end{eqnarray*}
and $\epsilon_n=\sqrt{n!/(2\lambda+1)_{n}}$. Here $\NN_0$ denotes the set of nonnegative integers. For details, see the next section. It is remarked that $\phi_0^{\lambda}(z)\equiv1$, and for $n\ge1$, $\phi_n^{0}(z)=z^n$.

We remark that for $\lambda$-analytic functions, there are no analog of Cauchy's theorem and also that of the Cauchy integral formula. Furthermore, roughly speaking, the product and the composition of $\lambda$-analytic functions are no longer $\lambda$-analytic.

The fundamental theory of the $\lambda$-Hardy spaces $H_{\lambda}^p(\DD)$ for $p\ge p_0$ was studied in $\cite{LL1}$, where
$$
p_0=\frac{2\lambda}{2\lambda+1}.
$$
Such a restriction on the exponent $p$ is due to the fact that $f\in H_{\lambda}^p(\DD)$ for these $p$ has a so-called $\lambda$-harmonic majorization (cf. \cite[Theorem 6.3]{LL1}). This phenomenon also occurs in the study of the Hardy spaces
associated with the Gegenbauer expansions in \cite{MS}, and in the pioneering work \cite{SW1} about the Hardy spaces on the half space $\RR_+^{d+1}$ for $d>1$, where the lower bound of $p$ is $(d-1)/d$.

In this paper we study several basic problems on the $\lambda$-Bergman
spaces $A^{p}_{\lambda}(\DD)$, and also on the $\lambda$-Hardy spaces $H_{\lambda}^p(\DD)$, such as boundedness of the Bergman projection, growth of functions, density, completeness, and the dual spaces of $A^{p}_{\lambda}(\DD)$ and $H_{\lambda}^p(\DD)$, and also characterization and interpolation of $A^{p}_{\lambda}(\DD)$.

%We also prove some inequalities on coefficients of functions in $A^{p}_{\lambda}(\DD)$, and several multiplier theorems of the spaces $A^{p}_{\lambda}(\DD)$ and $H_{\lambda}^p(\DD)$.

%A general description of multipliers with respect to $\lambda$-analytic functions on $\DD$ is given as follows. Suppose that $X$ and $Y$ are two spaces of $\lambda$-analytic functions on $\DD$. A sequence $\rho=\{\rho_n\}$ of complex numbers is said to be a (coefficient) multiplier from $X$ to $Y$ if the function $T_{\rho}f(z):=\sum_{n=0}^{\infty}\rho_nc_{n}\phi_{n}^{\lambda}(z)$ is in $Y$ whenever  $f(z)=\sum_{n=0}^{\infty}c_{n}\phi_{n}^{\lambda}(z)$ is in $X$. The collection of multipliers from $X$ to $Y$ will be denoted by $(X,Y)$.

There are rich theories of the Hardy spaces, the Bergman spaces, and other spaces of (usual) analytic functions on the unit disk and even more general domains in the plane or in higher-dimensional complex spaces.
See \cite{Du2,Ga1,Ko1,Zy1} for the Hardy spaces, and \cite{DS1,HKZ1,Zhu1,Zhu2} for the Bergman spaces.

The paper is organized as follows. In Section 2 we recall some basic knowledge about $\lambda$-analytic functions and $\lambda-$harmonic functions on the disk $\DD$. The Bergman projection associated to the $\lambda$-Bergman spaces is introduced in Section 3 and is proved to be bounded from $L_{\lambda}^{p}(\DD)$ into $A^{p}_{\lambda}(\DD)$ for $1<p<\infty$. In Section 4 we obtain the growth estimates of the $p$-means $M_p(f;r)$ of functions in $A^{p}_{\lambda}(\DD)$ and the point estimates of functions in both $H_{\lambda}^p(\DD)$ and $A^{p}_{\lambda}(\DD)$. Section 5 is devoted to density, completeness, and duality of the $\lambda$-Hardy and $\lambda$-Bergman spaces, and Section 6 to a characterization by the operator $D_z$ of the space $A^{p}_{\lambda}(\DD)$ for $1\le p<\infty$, and also an interpolation theorem of $A^{p}_{\lambda}(\DD)$.

%A multiplier theorem from $H^{p}_{\lambda}(\DD)$ into $H^{q}_{\lambda}(\DD)$ is proved in Section 7, and a multiplier theorem between $H^{p}_{\lambda}(\DD)$ and $A^{p}_{\lambda}(\DD)$ is proved in Section 8. In Section 9 we prove a Hausdorff-Young type theorem and a Hardy-Littlewood type theorem for the $\lambda$-Bergman spaces, and in Section 10, we obtain some necessary conditions and also some sufficient conditions on coefficient multipliers from $A^{p}_{\lambda}(\DD)$ into $A^{q}_{\lambda}(\DD)$.

%Various results on coefficient multipliers for the Hardy spaces and the Bergman spaces of usual analytic functions on $\DD$ can be found in \cite{Bla1,BRV1,Du1,Du2,DS1,Lou1,MZ1,Vu2,Wo1}, where further references are given.

For $0<p<\infty$, we denote by $L^p(\partial\DD;dm_{\lambda})$, or simply by $L_{\lambda}^{p}(\partial\DD)$,
the space of measurable functions $f$ on $\partial\DD$ satisfying
$\|f\|_{L_{\lambda}^{p}(\partial\DD)}:=\left(\int_{-\pi}^{\pi}|f(e^{i\theta})|^p
dm_{\lambda}(\theta)\right)^{1/p}<\infty$,
and for $p=\infty$, $L^{\infty}(\partial\DD;dm_{\lambda})=L^{\infty}(\partial\DD)$ as usual, with norm $\|f\|_{L^{\infty}(\partial\DD)} ={\rm esssup}_{\theta} |f(e^{i\theta})|$. In addition, ${\frak B}_{\lambda}(\partial\DD)$ denotes
the space of Borel measures $d\nu$ on $\partial\DD$ for which
$\|d\nu\|_{{\frak B_{\lambda}(\partial\DD)}}
=\tilde{c}_{\lambda}\int_{-\pi}^{\pi}|\sin\theta|^{2\lambda}|d\nu(\theta)|$
are finite.  Throughout the paper, the notation ${\mathcal{X}}\lesssim {\mathcal{Y}}$ or ${\mathcal{Y}}\gtrsim {\mathcal{X}}$ means that ${\mathcal{X}}\le c{\mathcal{Y}}$ for some positive constant $c$ independent of variables, functions, etc., and ${\mathcal{X}}\asymp {\mathcal{Y}}$ means that both ${\mathcal{X}}\lesssim {\mathcal{Y}}$ and ${\mathcal{Y}}\lesssim {\mathcal{X}}$ hold.

\section{$\lambda$-analytic and $\lambda$-harmonic functions}

Most materials of this section come from \cite{LL1}. This topic is motivated by C. Dunkl's work \cite{Dun3}, where he built up a framework associated with the dihedral group
$G=D_k$ on the disk $\DD$. Our work here and also that in \cite{LL1} focuses on the special case with $G=D_1$ having
the reflection $z\mapsto\overline{z}$ only, to find possibilities to develop a deep theory of associated function spaces. We note that C. Dunkl has a general theory named after him associated with reflection-invariance on the Euclidean spaces, see \cite{Dun1}, \cite{Dun2} and \cite{Dun4} for example.

\begin{proposition} \label{orthonirmal-basis-a} {\rm(\cite{Dun3,Dun4})}
For $n\ge1$ and $z=re^{i\theta}$, we have
\begin{align*}
\phi_{n}^{\lambda}(z)&=\epsilon_nr^{n}\left[\frac{n+2\lambda}{2\lambda}P_{n}^{\lambda}
(\cos\theta)+i\sin\theta P_{n-1}^{\lambda+1}(\cos\theta)\right],\\
\bar{z}\overline{\phi_{n-1}^{\lambda}(z)}
&=\epsilon_{n-1}r^{n}\left[\frac{n}{2\lambda}P_{n}^{\lambda}(\cos\theta)-
 i\sin\theta P_{n-1}^{\lambda+1}(\cos\theta)\right],
\end{align*}
where $P_{n}^{\lambda}(t)$, $n\in\NN_0$, are the Gegenbauer
polynomials (cf. \cite{Sz}), and $P_{-1}^{\lambda+1}=0$. Moreover, the system
\begin{align*}
\{\phi_{n}^{\lambda}(e^{i\theta}): \,\, n\in\NN_0\}\cup
\{e^{-i\theta}\phi_{n-1}^{\lambda}(e^{-i\theta}): \,\,n\in\NN\}
\end{align*}
is an orthonormal basis of the Hilbert space
$L_{\lambda}^2(\partial\DD)$.
\end{proposition}

It follows from \cite[Proposition 2.2]{LL1} that,
$\phi_{n}^{\lambda}(z)$ ($n\in\NN_0$) is $\lambda$-analytic, and $\bar{z}\overline{\phi_{n-1}^{\lambda}}(z)$ ($n\in\NN$)
is $\lambda$-harmonic; moreover,
\begin{align}
D_{z}\phi_{n}^{\lambda}(z)=\sqrt{n(n+2\lambda)}\phi_{n-1}^{\lambda}(z) \quad \hbox{and} \quad
D_{z}(z\phi_{n-1}^{\lambda}(z))=(n+\lambda)\phi_{n-1}^{\lambda}(z).\label{Tzphi-2}
\end{align}

The function $\phi_{n}^{\lambda}(z)$ ($n\in\NN_0$) has a closed representation as given in \cite[(29)]{LL1}, that is,
\begin{eqnarray} \label{A-operator-4}
\epsilon_n\phi_n(z)
=2^{2\lambda+1}\tilde{c}_{\lambda-1/2}\int_{0}^{1}(sz+(1-s)\bar{z})^n(1-s)^{\lambda-1}s^{\lambda}ds.
\end{eqnarray}
It is easy to see that
\begin{eqnarray}\label{phi-bound-1}
|\phi_n(z)|\le\epsilon_n^{-1}|z|^n\asymp (n+1)^{\lambda}|z|^n.
\end{eqnarray}

In what follows, we write $\phi_n(z)=\phi_n^{\lambda}(z)$ for simplicity.

We say $f$ to be a $\lambda$-analytic polynomial, or a $\lambda$-harmonic polynomial, if it is a finite linear combination of elements in the system $\{\phi_n(z): \,n\in\NN_0\}$, or in the system
\begin{align}\label{harmonic-polynomial-1}
\{\phi_n(z): \,\, n\in\NN_0\}\cup\{\bar{z}\overline{\phi_{n-1}(z)}: \,\, n\in\NN\}.
\end{align}

The $\lambda$-Cauchy kernel $C(z,w)$ and the $\lambda$-Poisson kernel $P(z,w)$, which reproduce, associated with the measure $dm_{\lambda}$ on the circle $\partial\DD$,
all $\lambda$-analytic polynomials
and $\lambda$-harmonic polynomials respectively, are given by (cf. \cite{Dun3})
\begin{align}
C(z,w)&=\sum_{n=0}^{\infty}\phi_{n}(z)\overline{\phi_{n}(w)},\label{Cauchy-kernel-2-1}\\
P(z,w)&=C(z,w)+\bar{z}w C(w,z).\label{Poisson-kernel-2-1}
\end{align}
The series in $C(z,w)$ is convergent absolutely for $zw\in\DD$ and uniformly for $zw$ in a compact subset of $\DD$, and from \cite{Dun3} (see \cite{LL1} also),
\begin{align}
C(z,w)&=\frac{1}{1-z\bar{w}}P_{0}(z,w), \qquad zw\in\DD, \label{Cauchy-kernel-2-2}\\
P(z,w)&=\frac{1-|z|^{2}|w|^{2}}{|1-z\bar{w}|^{2}}P_{0}(z,w),  \qquad zw\in\DD, \nonumber
\end{align}
where
\begin{align} \label{Poi-0}
P_{0}(z,w)&=\frac{1}{|1-zw|^{2\lambda}}{}_2\!F_{1}\Big({\lambda,\lambda
\atop
  2\lambda+1};\frac{4({\rm Im} z)({\rm Im}
  w)}{|1-zw|^{2}}\Big)\nonumber\\
&=\frac{1}{|1-z\bar{w}|^{2\lambda}}{}_2\!F_{1}\Big({\lambda,\lambda+1
\atop
  2\lambda+1};-\frac{4({\rm Im} z)({\rm Im}
  w)}{|1-z\bar{w}|^{2}}\Big),
\end{align}
and ${}_2\!F_{1}[a,b;c;t]$ is the Gauss hypergeometric function.

A $\lambda$-harmonic function on $\DD$ has a series representation
in terms of the system (\ref{harmonic-polynomial-1}). Precisely by \cite[Theorem 3.1]{LL1}, if $f$ is a $\lambda$-harmonic function on $\DD$, then there are two
sequences $\{c_n\}$ and $\{\tilde{c}_n\}$ of complex numbers, such
that
\begin{eqnarray*}
f(z)=\sum_{n=0}^{\infty}
c_n \phi_{n}(z)+
  \sum_{n=1}^{\infty} \tilde{c}_n \bar{z}\overline{\phi_{n-1}(z)}
\end{eqnarray*}
for $z\in\DD$; furthermore, the two sequences above are given by
\begin{align*}
c_n&=\lim_{r\rightarrow1-}\int_{-\pi}^{\pi}f(re^{i\varphi})\overline{\phi_{n}(e^{i\varphi})}
dm_{\lambda}(\varphi),\\
\tilde{c}_n&=\lim_{r\rightarrow1-}\int_{-\pi}^{\pi}f(re^{i\varphi})
e^{i\varphi}\phi_{n-1}(e^{i\varphi})dm_{\lambda}(\varphi),
\end{align*}
and satisfy the condition that, for each real $\gamma$, the series
$\sum_{n\ge1}n^{\gamma}(|c_n|+|\tilde{c}_n|)r^n$ converges uniformly
for $r$ in every closed subset of $[0,1)$.

As stated in the first section, a $\lambda$-analytic
function $f$ on $\DD$ has a series representation in terms of the system $\{\phi_n(z): \,n\in\NN_0\}$, as in (\ref{anal-series-1-1}); and moreover, $f$ could also be characterized by a Cauchy-Riemann type system. We summarize these conclusions as follows.

\begin{proposition} \label{anal-thm} {\rm(\cite[Theorem 3.7]{LL1})}
For a $C^2$ function $f=u+iv$ defined on $\DD$,
the following statements are equivalent:

{\rm (i)}  $f$ is $\lambda$-analytic;

{\rm (ii)} $u$ and $v$ satisfy the generalized Cauchy-Riemann
system
\begin{eqnarray*}
\left \{\partial_{x}u=D_{y}v,\atop  D_{y}u=-\partial_{x}v,\right.
\end{eqnarray*}
where
\begin{align*}
D_y u(x,y)=\partial_y u(x,y)+\frac{\lambda}{y}\left[u(x,y)-u(x,-y)\right].
\end{align*}

{\rm (iii)}  $f$ has the series representation
\begin{eqnarray} \label{anal-series-1}
f(z)=\sum_{n=0}^{\infty}c_{n}\phi_{n}(z), \qquad |z|<1,
\end{eqnarray}
where
$$
c_{n}=\lim_{r\rightarrow 1-}\int_{-\pi}^{\pi}f(re^{i\varphi})
\overline{\phi_{n}(e^{i\varphi})}dm_{\lambda}(\varphi).
$$
Moreover, for each real $\gamma$, the series
$\sum_{n\ge1}n^{\gamma}|c_n|r^n$ converges uniformly for $r$ in
every closed subset of $[0,1)$.
\end{proposition}

%The $\lambda$-analytic functions and the usual analytic functions are connected by an integral transform. By \cite[Theorem 3.9]{LL1},
%a function $f$ defined on $\DD$ is $\lambda-$analytic if and only if there exists an analytic function $f_0$ in
%$\DD$ such that $f(z)=A_{\lambda}f_0(z)$, where
%\begin{eqnarray*} \label{A-operator-2}
%A_{\lambda}f_0(z)=\frac{\tilde{c}_{\lambda-1/2}}{4^{\lambda}\,i}\int_{|\xi|=1}f_0\left(x+iy\,{\rm Re}\,\xi\right)
%\left(1+\xi^{-1}\right)^2 \left|1-\xi^2\right|^{2\lambda-1}d\xi
%\end{eqnarray*}
%with $z=x+iy$.
%The inversion of $A_{\lambda}$ is given by, for $|w|<r<1$,
%\begin{eqnarray*}
%A_{\lambda}^{-1}f(w)=\frac{\tilde{c}_{\lambda}}{2^{2\lambda}i}\oint_{|z|=r}
%\frac{f(z)}{(\bar{z}-w)^{\lambda}(z-w)^{\lambda+1}}|z-\bar{z}|^{2\lambda}dz.
%\end{eqnarray*}

The $\lambda$-Poisson integral of $f\in L_{\lambda}^{1}(\partial\DD)$ is defined by
\begin{align} \label{Poisson-integral-1}
P(f;z)=\int_{-\pi}^{\pi}f(e^{i\varphi})P(z,e^{i\varphi})
dm_{\lambda}(\varphi),\qquad z=re^{i\theta}\in\DD,
\end{align}
and that of a measure $d\nu\in{\frak B}_{\lambda}(\partial\DD)$ by
\begin{eqnarray} \label{Poisson-Borel-1}
P(d\nu;z)=\tilde{c}_{\lambda}\int_{-\pi}^{\pi}P(z,e^{i\varphi})
|\sin\varphi|^{2\lambda}d\nu(\varphi),\qquad z=re^{i\theta}\in\DD.
\end{eqnarray}

\begin{proposition} \label{Poi-property-1} {\rm(\cite[Propositions 2.4 and 2.5]{LL1})}
Let $f\in L_{\lambda}^{1}(\partial\DD)$. Then

 {\rm (i)} the $\lambda$-Poisson integral $u(x,y)=P(f;z)$ ($z=x+iy$) is
$\lambda$-harmonic
in $\DD$;

{\rm (ii)} if write $P_r(f;\theta)=P(f;re^{i\theta})$, we have the ``semi-group" property
$$
P_s(P_rf;\theta)=P_{s\,r}(f;\theta), \qquad 0\le s,r<1;
$$

{\rm (iii)} $P(f;z) \ge 0$ if $f \ge 0$;

{\rm (iv)} for $f\in X=L_{\lambda}^{p}(\partial\DD)$ ($1\le p\le\infty$), or $C(\partial\DD)$,
$\|P_r(f;\cdot)\|_{X}\le \|f\|_{X}$;

{\rm (v)} for $f\in X=L_{\lambda}^{p}(\partial\DD)$ ($1\le p<\infty$), or
$C(\partial\DD)$, $\lim_{r\rightarrow1-}\|P_r(f;\cdot)-f\|_X=0$;

{\rm (vi)} the conclusions in (i)-(iii) are true also for $P_r(d\nu;\cdot)$ in place of $P_r(f;\cdot)$, if
$d\nu\in{\frak B}_{\lambda}(\partial\DD)$; moreover, $\|P_r(d\nu;\cdot)\|_{L_{\lambda}^{1}(\partial\DD)}\le \|d\nu\|_{\frak B_{\lambda}(\partial\DD)}$.
\end{proposition}

The following theorem asserts the existence of boundary values of functions in the $\lambda$-Hardy spaces.

\begin{theorem} \label{Hardy-boundary-value-a} {\rm (\cite[Theorem 6.6]{LL1})}
Let $p\ge p_0$ and $f\in H_{\lambda}^p(\DD)$. Then for almost every $\theta\in[-\pi,\pi]$, $\lim f(r
e^{i\varphi})=f(e^{i\theta})$  exists as $r e^{i\varphi}$ approaches to the point
$e^{i\theta}$ nontangentially, and if $p_0<p<\infty$, then
\begin{align}\label{Hp-norm-convergence-1}
\lim_{r\rightarrow1-}\int_{-\pi}^{\pi}|f(r
e^{i\theta})-f(e^{i\theta})|^{p} dm_{\lambda}(\theta)=0
\end{align}
and $\|f\|_{H^{p}_{\lambda}}\asymp\left(\int_{-\pi}^{\pi}|f(e^{i\theta})|^{p}
dm_{\lambda}(\theta)\right)^{1/p}$.
%{\rm (iv)} \  For $p >p_0$, $F\in H^{p}(\DD)$ if and only if
%$N_{\Omega}F\in L^{p}(\SS^1,dm_{\lambda})$, where
%$N_{\Omega}F(\theta)=\sup_{r e^{i\varphi}\in\,\Omega(\theta)} |F(r
%e^{i\varphi})|$, and moreover
%   $$\|F\|_{H^{p}_{\lambda}}\leq \|N_{\Omega}F\|_{p}\leq c\|F\|_{H^{p}_{\lambda}};$$
%
%{\rm (v)} Suppose $p >p_0$ and $p_1\geq p_0$, $F(z)\in
%H^{p}(\DD)$ and $F(e^{i\theta}) \in L^{p_1}(\SS^1,dm_{\lambda})$.
%Then $F(z) \in H^{p_1}(\DD)$.
\end{theorem}

By \cite[Theorems 6.6 and 6.8]{LL1}, we have

\begin{theorem} \label{Hp-Lp-thm-1}
Let $1\le p\le\infty$. If $f\in H_{\lambda}^p(\DD)$, then its boundary value $f(e^{i\theta})$ is in $L_{\lambda}^{p}(\partial\DD)$, and $f$ can be recovered from $f(e^{i\theta})$ by the $\lambda$-Poisson integral, namely,
\begin{align} \label{Poisson-1}
f(z)=\int_{-\pi}^{\pi}f(e^{i\varphi})P(z,e^{i\varphi})
dm_{\lambda}(\varphi),\qquad z\in\DD.
\end{align}
Moreover, $\|f\|_{H_{\lambda}^p}=\|f\|_{L_{\lambda}^{p}(\partial\DD)}$, and $f(e^{i\theta})$ satisfies the condition
\begin{align}\label{coefficient-2-2}
\int_{-\pi}^{\pi}f(e^{i\varphi})e^{i\varphi}\phi_{n-1}(e^{i\varphi})dm_{\lambda}(\varphi)=0,\qquad n=1,2,\dots,
\end{align}
and if $f$ has the expansion $f(z)=\sum_{n=0}^{\infty}c_{n}\phi_{n}(z)$, then
\begin{align}\label{coefficient-2-3}
c_n=\int_{-\pi}^{\pi}f(e^{i\varphi})\overline{\phi_{n}(e^{i\varphi})}
dm_{\lambda}(\varphi),\qquad n=0,1,\cdots.
\end{align}
Conversely, the $\lambda$-Poisson integral of a function $f$ in $L_{\lambda}^{p}(\partial\DD)$ satisfying (\ref{coefficient-2-2})
is an element in $H_{\lambda}^p(\DD)$.
\end{theorem}

For the Hardy spaces on the upper half plane $\RR_+^2=\RR\times(0,\infty)$ associated to the Dunkl operators $D_{z}$ and $D_{\bar{z}}$, see \cite{LL2}.

\section{Preliminaries to the $\lambda$-Bergman spaces}

We define the function $K_{\lambda}(z,w)$ by
\begin{eqnarray*}
K_{\lambda}(z,w)=\sum_{n=0}^{\infty}
\frac{n+\lambda+1}{\lambda+1}\phi_{n}(z)\overline{\phi_{n}(w)}.
\end{eqnarray*}
By (\ref{phi-bound-1}), the series in $K_{\lambda}(z,w)$ is convergent absolutely for $zw\in\DD$ and uniformly for $zw$ in a compact subset of $\DD$. Moreover we have

\begin{proposition} \label{reproducing-Bergman-a}
For fixed $w\in\overline{\DD}$ the function $z\mapsto K_{\lambda}(z,w)$ is $\lambda$-analytic in $\DD$, and the function $w\mapsto K_{\lambda}(z,w)$ reproduces all functions $f\in A^{1}_{\lambda}(\DD)$, that is,
\begin{eqnarray}\label{reproducing-Bergman-1}
f(z)=\int_{\DD}f(w)K_{\lambda}(z,w)d\sigma_{\lambda}(w),\qquad z\in\DD.
\end{eqnarray}
Conversely, if $f\in L_{\lambda}^{1}(\DD)$ satisfies (\ref{reproducing-Bergman-1}), then $f$ is $\lambda$-analytic in $\DD$, and in particular $f\in A^{1}_{\lambda}(\DD)$.
\end{proposition}

\begin{proof}
It is easy to see, from (\ref{A-operator-4}), that
\begin{align}\label{phi-bound-2}
\left|\partial_{\bar{z}}\phi_n(z)\right|\le\epsilon_n^{-1}n|z|^{n-1}\asymp n^{\lambda+1}|z|^{n-1},
\end{align}
which shows that taking termwise differentiation $\partial_{\bar{z}}$ in $\DD$ to $K_{\lambda}(z,w)$ is legitimate. Therefore $D_{\bar{z}}K_{\lambda}(z,w)=0$ since $D_{\bar{z}}\phi_n(z)=0$.

For $f\in A^{1}_{\lambda}(\DD)$ and $0<r<1$, from (\ref{anal-series-1}) we have $f(rw)=\sum_{k=0}^{\infty}c_{k}r^k\phi_{k}(w)$ ($|w|<1$).
In view of (\ref{phi-bound-1}),
the last assertion in Proposition \ref{anal-thm}(iii) implies that termwise integration of $\overline{\phi_{n}(w)}f(rw)$ over $\DD$ is legitimate, and since, by Proposition \ref{orthonirmal-basis-a},
\begin{align}\label{Bergman-normalization-1}
\int_{\DD}\left|\phi_{n}(w)\right|^2\,d\sigma_{\lambda}(w)
=\frac{\lambda+1}{n+\lambda+1},
\end{align}
it follows that
$\int_{\DD}\overline{\phi_{n}(w)}f(rw)\,d\sigma_{\lambda}(w)=\frac{\lambda+1}{n+\lambda+1}r^nc_n$.
Making change of variables $w\mapsto w/r$, one has
$$
\int_{\DD_r}\overline{\phi_{n}(w)}f(w)\,d\sigma_{\lambda}(w)=\frac{\lambda+1}{n+\lambda+1}r^{2n+2\lambda+2}c_n,
$$
where $\DD_r=\{w:\,|w|<r\}$, and then, letting $r\rightarrow1-$ yields
\begin{align}\label{coefficient-2-1}
\int_{\DD}\overline{\phi_{n}(w)}f(w)\,d\sigma_{\lambda}(w)=\frac{\lambda+1}{n+\lambda+1}c_n.
\end{align}

Again by (\ref{phi-bound-1}), $\sum_{n=0}^{\infty}(n+\lambda+1)|\phi_{n}(z)\overline{\phi_{n}(w)}|\lesssim\sum_{n=0}^{\infty}(n+1)^{2\lambda+1}|z|^n$ for all $w\in\overline{\DD}$, which implies that termwise integration of
$f(w)K_{\lambda}(z,w)$ over $\DD$ is also legitimate. This, together with (\ref{coefficient-2-1}), proves (\ref{reproducing-Bergman-1}).

The final assertion in the proposition is verified by the same procedure as above.
\end{proof}

The function $K_{\lambda}(z,w)$ is called the $\lambda$-Bergman kernel on the disk $\DD$.

\begin{corollary}\label{reproducing-Bergman-b}
If $f$ is $\lambda$-analytic in $\DD$, then for $r\in(0,1)$,
\begin{eqnarray}\label{reproducing-Bergman-2}
f(z)=r^{-2\lambda-2}\int_{\DD_r}f(w)K_{\lambda}(z/r,w/r)d\sigma_{\lambda}(w),\qquad z\in\DD_r,
\end{eqnarray}
where $\DD_r=\{z:\,|z|<r\}$.
Conversely, if $f$ is locally integrable on $\DD$ associated with the measure $d\sigma_{\lambda}$ and satisfies (\ref{reproducing-Bergman-2}) for all $r\in(0,1)$, then $f$ is $\lambda$-analytic in $\DD$.
\end{corollary}

\begin{proof}
For $z\in\DD_r$, set $z'=z/r$. Applying (\ref{reproducing-Bergman-1}) to the function $z'\mapsto f(rz')$ one has
$$
f(rz')=\int_{\DD}f(rw)K_{\lambda}(z',w)d\sigma_{\lambda}(w),
$$
and then, making change of variables $w\mapsto w'/r$ yields (\ref{reproducing-Bergman-2}). The second part of the corollary follows from (\ref{reproducing-Bergman-2}) and the final assertion in Proposition \ref{reproducing-Bergman-a}.
\end{proof}

We define the operator $P_{\lambda}$ by
\begin{align}\label{Bergman-projection-1}
(P_{\lambda}f)(z)=\int_{\DD}f(w)K_{\lambda}(z,w)d\sigma_{\lambda}(w),\qquad z\in\DD.
\end{align}

\begin{proposition} \label{reproducing-Bergman-c}
For $f\in L_{\lambda}^{1}(\DD)$, the integral on the right hand side of (\ref{Bergman-projection-1}) is well defined for $z\in\DD$ and defines a $\lambda$-analytic function in $\DD$.
\end{proposition}

Indeed as in the proof of Proposition \ref{reproducing-Bergman-a}, for fixed $z\in\DD$, by (\ref{phi-bound-1}) termwise integration of
$f(w)K_{\lambda,\alpha}(z,w)$ over $\DD$ with respect to the measure $d\sigma_{\lambda}(w)$ is legitimate, and by Proposition \ref{anal-thm}(iii) the resulting series represents a $\lambda$-analytic function in $\DD$.

\begin{lemma} \label{Cauchy-estimate-a}  {\rm(\cite[Theorem 4.2]{LL1})}
For $|zw|<1$, we have
\begin{eqnarray*}
|C(z,w)|\lesssim \frac{|1-z\bar{w}|^{-1}}{\big(|1-zw|+|1-z\bar{w}|\big)^{2\lambda}}\ln\left(\frac{|1-z\bar{w}|^2}{|1-zw|^2}+2\right).
\end{eqnarray*}
\end{lemma}

\begin{lemma}\label{Bergman-kernel-a}
For $|zw|<1$,
\begin{align}\label{Bergman-kernal-2}
|K_{\lambda}(z,w)|\lesssim \frac{|1-z\bar{w}|^{-1}}{\big(|1-zw|+|1-z\bar{w}|\big)^{2\lambda}}
\left(\frac{1}{|1-z\overline{w}|}+\frac{1}{|1-zw|}\right).
\end{align}
\end{lemma}

\begin{proof}
For $|zw|<1$ and $z=re^{i\theta}$, one has
\begin{align*}
(\lambda+1)K_{\lambda}(z,w)=r\frac{d}{dr}\left[C(z,w)\right]+(\lambda+1)C(z,w).
\end{align*}
By Lemma \ref{Cauchy-estimate-a} and on account of the fact $\frac{1}{s}\ln\left(\frac{s^{2}}{t^2}+2\right)\lesssim\frac{1}{s}+\frac{1}{t}$ for $s,t>0$, it suffices to show that $\left|r\frac{d}{dr}\left[C(z,w)\right]\right|$ has the same upper bound as in (\ref{Bergman-kernal-2}).

Put
\begin{align*}
A=\frac{4({\rm Im} z)({\rm Im} w)} {|1-zw|^{2}},\qquad \tilde{A}=-\frac{4({\rm Im} z)({\rm Im}
  w)}{|1-z\bar{w}|^{2}}.
\end{align*}
It is noted that
\begin{align*}
r\frac{d}{dr}\left[\frac{1}{1-z\bar{w}}\right]&=\frac{z\bar{w}}{(1-z\bar{w})^2},\qquad z=re^{i\theta},\\
r\frac{d}{dr}\left[\frac{1}{|1-zw|^{2\lambda}}\right]
&=2\lambda\frac{{\mathrm Re}\,(zw)-|zw|^2}{|1-zw|^{2\lambda+2}},\qquad z=re^{i\theta},\\
r\frac{d}{dr}A&=A\left(1+2\frac{{\mathrm Re}\,(zw)-|zw|^2}{|1-zw|^{2}}\right)
=A\frac{1-|zw|^2}{|1-zw|^{2}},\qquad z=re^{i\theta}.
\end{align*}

For
$|1-z\bar{w}|^2\le2|1-zw|^2$, we have $-1\le A<1$ since $1-A=|1-z\bar{w}|^2/|1-zw|^2$, so that
\begin{align*}
\left|{}_2\!F_{1}\left(\lambda,\lambda;2\lambda+1;A\right)\right|
\lesssim 1.
\end{align*}
It follows from \cite[2.1(7)]{Er} that
\begin{align*}
r\frac{d}{dr}\left[{}_2\!F_{1}\left(\lambda,\lambda;2\lambda+1;A\right)\right]
=\frac{\lambda^2}{2\lambda+1}{}_2\!F_{1}\left(\lambda+1,\lambda+1;2\lambda+2;A\right)r\frac{d}{dr}A,
\end{align*}
and successively using \cite[2.1(7) and 2.1(23)]{Er},
\begin{align} \label{Gauss-1}
{}_2\!F_{1}(a,b;a+b;t)&=1+\frac{ab}{a+b}\int_0^t(1-s)^{-1}{}_2\!F_{1}(a,b;a+b+1;s)dt \nonumber\\
&\asymp1+\ln\frac{1}{1-t}\asymp\ln\left(\frac{1}{1-t}+2\right),\qquad -1\le t<1.
\end{align}
Thus we have
\begin{align*}
\left|r\frac{d}{dr}\left[{}_2\!F_{1}\left(\lambda,\lambda;2\lambda+1;A\right)\right]\right|
\lesssim \frac{1}{|1-zw|}\ln\left(\frac{|1-zw|^2}{|1-z\bar{w}|^2}+2\right)
\lesssim \frac{1}{|1-z\overline{w}|},\qquad z=re^{i\theta},
\end{align*}
where the last inequality is based on the fact $s\ln\left(s^{-2}+2\right)\lesssim1$ for $s\in(0,2]$.
Now from (\ref{Cauchy-kernel-2-2}) and the first equality in (\ref{Poi-0}) we get, for
$|1-z\bar{w}|^2\le2|1-zw|^2$,
\begin{align}\label{Cauchy-inequality-2}
\left|r\frac{d}{dr}\left[C(z,w)\right]\right|
\lesssim &\frac{1}{|1-z\overline{w}|^2}\frac{1}{|1-zw|^{2\lambda}}
\asymp\frac{|1-z\overline{w}|^{-2}}{(|1-z\overline{w}|+|1-zw|)^{2\lambda}},\qquad z=re^{i\theta}.
\end{align}

If $2|1-zw|^2<|1-z\bar{w}|^2$, we have $1/2<\tilde{A}<1$ since $1-\tilde{A}=|1-zw|^2/|1-z\bar{w}|^2$, and then, (\ref{Gauss-1}) shows
\begin{align*}
\left|{}_2\!F_{1}\left(\lambda,\lambda+1;2\lambda+1;\tilde{A}\right)\right|
\lesssim \ln\left(\frac{|1-z\bar{w}|^2}{|1-zw|^2}+2\right).
\end{align*}
Again using \cite[2.1(7)]{Er} gives
\begin{align*}
r\frac{d}{dr}\left[{}_2\!F_{1}\left(\lambda,\lambda+1;2\lambda+1;\tilde{A}\right)\right]
=\frac{\lambda(\lambda+1)}{2\lambda+1}{}_2\!F_{1}\left(\lambda+1,\lambda+2;2\lambda+2;\tilde{A}\right)r\frac{d}{dr}\tilde{A},
\end{align*}
but by \cite[2.1(23)]{Er},
\begin{align*}
{}_2\!F_{1}\left(\lambda+1,\lambda+2;2\lambda+2;\tilde{A}\right)
=(1-\tilde{A})^{-1}{}_2\!F_{1}\left(\lambda+1,\lambda;2\lambda+2;\tilde{A}\right),
\end{align*}
so that
\begin{align*}
\left|r\frac{d}{dr}\left[{}_2\!F_{1}\left(\lambda,\lambda+1;2\lambda+1;\tilde{A}\right)\right]\right|
\lesssim &\frac{|1-z\bar{w}|^2}{|1-zw|^2}\frac{1-|zw|^2}{|1-z\bar{w}|^2}
=\frac{1-|zw|^2}{|1-zw|^2},\qquad z=re^{i\theta}.
\end{align*}
Now from (\ref{Cauchy-kernel-2-2}) and the second equality in (\ref{Poi-0}) we get, for
$2|1-zw|^2<|1-z\bar{w}|^2$,
\begin{align}\label{Cauchy-inequality-3}
\left|r\frac{d}{dr}\left[C(z,w)\right]\right|
\lesssim &\frac{1}{|1-z\overline{w}|^{2\lambda+2}}\ln\left(\frac{|1-z\bar{w}|^2}{|1-zw|^2}+2\right)
+\frac{1}{|1-z\overline{w}|^{2\lambda+1}}\frac{1}{|1-zw|} \nonumber\\
\lesssim &\frac{|1-z\overline{w}|^{-1}|1-zw|^{-1}}{(|1-z\overline{w}|+|1-zw|)^{2\lambda}},\qquad z=re^{i\theta}.
\end{align}

Combining (\ref{Cauchy-inequality-2}) and (\ref{Cauchy-inequality-3}) shows that $\left|r\frac{d}{dr}\left[C(z,w)\right]\right|$ has an upper bound like that in (\ref{Bergman-kernal-2}), as disired.
\end{proof}

\begin{theorem}\label{Bergman-projection-a}
For $1<p<\infty$, the operator $P_{\lambda}$ is bounded from $L_{\lambda}^{p}(\DD)$ onto $A^{p}_{\lambda}(\DD)$.
\end{theorem}

\begin{proof}
We shall show that, for $\alpha=1/p$ and $\alpha=1/p'$,
\begin{align}\label{Bergman-projection-boundedness-1}
\int_{\mathbb{D}}|K_{\lambda}(z,w)|(1-|w|^2)^{-\alpha}d\sigma_{\lambda}(w) & \lesssim(1-|z|^2)^{-\alpha}, \qquad z\in\DD.
\end{align}
Thus by Schur's theorem (cf. \cite[p. 52]{Zhu1}), $P_{\lambda}$ is bounded from $L_{\lambda}^{p}(\DD)$ to $A^{p}_{\lambda}(\DD)$.

For $z=re^{i\theta}$, $w=se^{i\varphi}\in\DD$, we have the following inequalities
\begin{align*}
&|1-z\overline{w}|\asymp1-rs+\left|\sin(\theta-\varphi)/2\right|,\\
&|1-z\overline{w}|+|1-zw|\gtrsim 1-rs+|\sin\theta|+|\sin\varphi|,
\end{align*}
and from (\ref{Bergman-kernal-2}),
\begin{align}\label{Bergman-kernal-3}
|K_{\lambda}(z,w)|\lesssim \Phi_{r,\theta}(s,\varphi)+\Phi_{r,\theta}(s,-\varphi),
\end{align}
where
\begin{align}\label{Bergman-kernal-4}
 \Phi_{r,\theta}(s,\varphi)
 =\frac{\left(1-rs+|\sin\theta|+|\sin\varphi|\right)^{-2\lambda}}{\left(1-rs+\left|\sin(\theta-\varphi)/2\right|\right)^{2}}.
\end{align}
We have
\begin{align*}
\int_{\mathbb{D}}|K_{\lambda}(z,w)|(1-|w|^2)^{-\alpha}d\sigma_{\lambda}(w)
\lesssim\int_0^1\int_{-\pi}^{\pi}\Phi_{r,\theta}(s,\varphi)(1-s)^{-\alpha}|\sin\varphi|^{2\lambda}d\varphi ds,
\end{align*}
since the contribution of $\Phi_{r,\theta}(s,-\varphi)$ to the integral is the same as that of $\Phi_{r,\theta}(s,\varphi)$.
Thus
\begin{align*}
\int_{\mathbb{D}}|K_{\lambda}(z,w)|(1-|w|^2)^{-\alpha}d\sigma_{\lambda}(w)
\lesssim\int_0^1\int_{-\pi}^{\pi}\frac{(1-s)^{-\alpha}}{\left(1-rs+\left|\sin(\theta-\varphi)/2\right|\right)^{2}}\,d\varphi ds,
\end{align*}
from which (\ref{Bergman-projection-boundedness-1}) is yields after elementary calculations.
\end{proof}

The proof above actually gives a slightly stronger conclusion, stated as follows.

\begin{corollary}\label{Bergman-projection-a-1}
For $1<p<\infty$, the operator $T$ defined by
$$
(Tf)(z)=\int_{\DD}|f(w)||K_{\lambda}(z,w)|d\sigma_{\lambda}(w)
$$
maps $L_{\lambda}^{p}(\DD)$ boundedly into itself.
\end{corollary}

\begin{corollary}\label{Bergman-projection-b}
There is an orthogonal projection from $L_{\lambda}^{2}(\DD)$ onto $A^{2}_{\lambda}(\DD)$. Such an operator is uniquely determined by (\ref{Bergman-projection-1}).
\end{corollary}

\begin{proof}
The existence and the uniqueness of the orthogonal projection from $L_{\lambda}^{2}(\DD)$ onto $A^{2}_{\lambda}(\DD)$ follows from the general theory of Hilbert spaces, since
$A^{2}_{\lambda}(\DD)$ is a closed subspace of $L_{\lambda}^{2}(\DD)$ (see Theorem \ref{completeness-Bergman-a} below). Moreover by Theorem \ref{Bergman-projection-a}, the operator $P_{\lambda}$ is bounded from $L_{\lambda}^{2}(\DD)$ into $A^{2}_{\lambda}(\DD)$, and by Proposition \ref{reproducing-Bergman-a}, $P_{\lambda}\left(L_{\lambda}^{2}(\DD)\right)=A^{2}_{\lambda}(\DD)$ and $P_{\lambda}^2=P_{\lambda}$. In addition, by Corollary \ref{Bergman-projection-a-1},
for $f,g\in L_{\lambda}^{2}(\DD)$ the interchange of order of the integration
\begin{align*}
\int_{\DD}\int_{\DD}f(w)K_{\lambda}(z,w)d\sigma_{\lambda}(w)\,\overline{g(z)}d\sigma_{\lambda}(z)
\end{align*}
is permitted, which implies that the adjoint $P_{\lambda}^*$ of $P_{\lambda}$ is itself.
\end{proof}

The operator $P_{\lambda}$ defined by (\ref{Bergman-projection-1}) is called the $\lambda$-Bergman projection on the disk $\DD$.

\section{Growth of functions in the $\lambda$-Hardy and $\lambda$-Bergman spaces}

\subsection{The $p$-means $M_p(f;r)$ of $\lambda$-analytic functions}

We recall several conclusions on the $\lambda$-Hardy spaces in \cite{LL1}.

\begin{theorem} \label{Hardy-mean-estimate-a} {\rm (\cite[Theorem 7.1]{LL1})}
Let $p_0\le p\le\ell\le+\infty$, $\delta=\frac{1}{p}-\frac{1}{\ell}$, and
let $f\in H^p_{\lambda}(\DD)$.
Then

{\rm (i)} \ for $p_0\le p\le\ell\le+\infty$, $M_{\ell}(f;r)\le c(1-r)^{-\delta(2\lambda+1)}\|f\|_{H_{\lambda}^p}$;

{\rm (ii)} \ for $p_0<p<\ell\le+\infty$,
$M_{\ell}(f;r)=o\left((1-r)^{-\delta(2\lambda+1)}\right)$;

{\rm (iii)} \ for $p_0<p<\ell\le+\infty$, $p\le k<+\infty$,
\begin{eqnarray*}
\left(\int_0^1(1-r)^{k\delta(2\lambda+1)-1}M_{\ell}(f;r)^kdr\right)^{1/k}\le
c\|f\|_{H_{\lambda}^p}.
\end{eqnarray*}
\end{theorem}

\begin{lemma}\label{monotonicity-integral-mean-2-a}
If $f$ is $\lambda$-harmonic in $\DD$ and $1\le p<\infty$, then its integral mean $r\mapsto M_p(f;r)$ is nondecreasing over $[0,1)$; and if $f$ is $\lambda$-analytic in $\DD$ and $p_0\le p<1$, then $M_p(f;r')\le2^{2/p-1}M_p(f;r)$ for $0\le r'<r<1$.
\end{lemma}

The first part above follows from \cite[Proposition 2.4(iv) and Lemma 3,3]{LL1}, and the second part from \cite[Theorem 6.4]{LL1}. For $p_0\le p<\infty$, we have the following nearly trivial inequality
\begin{align}\label{analytic-mean-estimate-0}
\int_0^1M_p(f;r)^p dr\lesssim\int_0^1M_p(f;r)^p r^{2\lambda+1}dr.
\end{align}
This is because
$\int_0^{1/2}M_p(f;r)^p dr \le 4\int_0^{1/2}M_p(f;1-r)^p dr=4\int_{1/2}^1M_p(f;r)^p dr$.

We now give some estimates of the $p$-means $M_p(f;r)$ of functions $f$ in the $\lambda$-Bergman spaces $A^{p}_{\lambda}(\DD)$.

\begin{proposition} \label{analytic-mean-estimate-a}
Let $p_0\le p\le l\le+\infty$, $\delta=\frac{1}{p}-\frac{1}{l}$, and let $f$ be a $\lambda$-analytic function on $\DD$. If for some $\beta\ge0$,
\begin{align}\label{analytic-mean-estimate-1}
M_p(f;r)\le c_0(1-r)^{-\beta},\qquad 0\le r<1,
\end{align}
then there exists some $c>0$ independent of $r$ and $c_0$ such that
\begin{align}\label{analytic-mean-estimate-2}
M_l(f;r)\le cc_0(1-r)^{-\delta(2\lambda+1)-\beta},\qquad 0\le r<1.
\end{align}
\end{proposition}

\begin{proof} For $p=\infty$ (\ref{analytic-mean-estimate-2}) is trivial. In what follows we assume $p_0\le p<\infty$.

For $0<s<1$ set $f_s(z)=f(sz)$. We apply Theorem \ref{Hardy-mean-estimate-a}(i) to $f_s$ to get
$$
M_l(f_s;r)\le c(1-r)^{-\delta(2\lambda+1)}\|f_s\|_{H_{\lambda}^p},\qquad 0\le r<1,
$$
but
\begin{align}\label{dilation-integral-mean-1}
\|f_s\|_{H_{\lambda}^p}=\sup_{0\le r <1}M_p(f_s;r)=\sup_{0\le r <1}M_p(f;sr)\le 2^{2/p}M_p(f;s),
\end{align}
where the last inequality is based on Lemma \ref{monotonicity-integral-mean-2-a}. Combining the above two estimates gives $M_l(f;sr)\le c'(1-r)^{-\delta(2\lambda+1)}M_p(f;s)$, which, together with (\ref{analytic-mean-estimate-1}), implies $M_l(f;r^2)\le c'c_0(1-r)^{-\delta(2\lambda+1)-\beta}$. Finally replacing $r$ by $\sqrt{r}$ proves (\ref{analytic-mean-estimate-2}).
\end{proof}

\begin{proposition} \label{Bergman-mean-estimate-a}
Let $p_0\le p<\infty$ and let $f$ be a $\lambda$-analytic function on $\DD$. If $f\in A^{p}_{\lambda}(\DD)$, then
$M_p(f;r)=o\left((1-r)^{-1/p}\right)$ as $r\rightarrow1-$, and $M_p(f;r)\lesssim(1-r)^{-1/p}\|f\|_{A_{\lambda}^{p}}$ for $0\le r<1$.
If for some $\alpha<1/p$, $M_p(f;r)\lesssim(1-r)^{-\alpha}$, then $f\in A^{p}_{\lambda}(\DD)$.
\end{proposition}

Indeed, for $f\in A^{p}_{\lambda}(\DD)$, one uses Lemma \ref{monotonicity-integral-mean-2-a} to get
$$
(1-r)M_p(f;r)^p\le 4\int_r^1M_p(f;s)^p ds\lesssim\int_r^1M_p(f;s)^p s^{2\lambda+1}ds\rightarrow0\quad
\hbox{as}\,\,r\rightarrow1-;
$$
and further by use of (\ref{analytic-mean-estimate-0}), $(1-r)M_p(f;r)^p\lesssim\|f\|_{A_{\lambda}^{p}}^p$.
This verifies the first part of the proposition. The second part is obvious.

The following theorem is a combination of Propositions \ref{analytic-mean-estimate-a} and \ref{Bergman-mean-estimate-a}.

\begin{theorem}\label{Bergman-mean-estimate-b}
Let $p_0\le p<\infty$, $p\le \ell\le\infty$, and $\delta=\frac{1}{p}-\frac{1}{\ell}$. Then for $f\in A^{p}_{\lambda}(\DD)$,
$M_{\ell}(f;r)=o\left((1-r)^{-\delta(2\lambda+1)-1/p}\right)$ as $r\rightarrow1-$, and
\begin{align*}
M_{\ell}(f;r)\lesssim(1-r)^{-\delta(2\lambda+1)-1/p}\|f\|_{A_{\lambda}^{p}},\qquad 0\le r<1.
\end{align*}
\end{theorem}

\subsection{Point-evaluation of functions in the $\lambda$-Hardy and $\lambda$-Bergman spaces}

We shall need the $\lambda$-harmonic majorizations of functions in $H_{\lambda}^p(\DD)$.

\begin{theorem} \label{majorization-a} {\rm (\cite[Theorem 6.3]{LL1}, $\lambda$-harmonic majorization)}
Let $p\ge p_0$ and $f\in H_{\lambda}^p(\DD)$.

  {\rm (i)} \  If $p >p_0,$ then there exists a nonnegative function $g(\theta)\in
L_{\lambda}^{p/p_0}(\partial\DD)$ such that for $z\in\DD$,
\begin{eqnarray*}
|f(z)|^{p_0}\leq P(g;z),
\end{eqnarray*}
and
\begin{eqnarray*}
\|f\|^{p_0}_{H^{p}_{\lambda}}\le\|g\|_{L_{\lambda}^{p/p_0}(\partial\DD)}\le2^{2-p_0}\|f\|^{p_0}_{H^{p}_{\lambda}}.
\end{eqnarray*}

{\rm (ii)} \ If $p=p_0,$ there exists a finite positive measure
  $d\nu\in {\frak B}_{\lambda}(\partial\DD)$ such that for $z\in\DD$,
\begin{eqnarray*}
|f(z)|^{p_0}\le P(d\nu;z),
\end{eqnarray*}
and
\begin{eqnarray*}
\|f\|^{p_0}_{H^{p_0}_{\lambda}}\le \|d\nu\|_{{\frak B}_{\lambda}(\partial\DD)}
\le2^{2-p_0}\|f\|^{p_0}_{H^{p_0}_{\lambda}}.
\end{eqnarray*}
\end{theorem}

We first consider the $\lambda$-harmonic function $u(x,y)$ in the disc $\DD$ satisfying the condition, for $1\le p<\infty$,
\begin{align}\label{harmonic-Lp-condition-1}
C_p(u):=\sup_{0\le r<1}\int_{-\pi}^{\pi}|u(r\cos\theta,r\sin\theta)|^p
dm_{\lambda}(\theta)<\infty.
\end{align}

\begin{lemma} \label{Poi-charact} {\rm (\cite[Corollary 3.5]{LL1}, characterization of $\lambda$-Poisson integrals)} \
Let $u(x,y)$ be $\lambda$-harmonic in $\DD$. Then

{\rm (i)} for $1<p<\infty$, $u(x,y)$ is the $\lambda$-Poisson
  integral $P(f;z)$ with $z=x+iy$ of some $f\in L_{\lambda}^{p}(\partial\DD)$ if and only if $C_p(u)<\infty$;

{\rm (ii)} $u(x,y)$ is the $\lambda$-Poisson
integral $P(d\nu;z)$ of a measure $d\nu\in {\frak B}_{\lambda}(\partial\DD)$ if and
only if $C_1(u)<\infty$;

{\rm (iii)} in parts (i) and (ii), $\|f\|_{L_{\lambda}^{p}(\partial\DD)}=C_p(u)^{1/p}$ for $1<p<\infty$, and $\|d\nu\|_{{\frak B_{\lambda}}}=C_1(u)$ for $p=1$.
\end{lemma}

Note that, in the above lemma, for $1<p<\infty$ the equality $\|f\|_{L_{\lambda}^{p}(\partial\DD)}=C_p(u)^{1/p}$ follows from Proposition \ref{Poi-property-1}(iv) and (v) immediately. For $p=1$, the Riesz representation theorem gives
\begin{align*}
\|d\nu\|_{{\frak B_{\lambda}}}
=\sup c_{\lambda}\left|\int_{-\pi}^{\pi}h(\theta)|\sin\theta|^{2\lambda}d\nu(\theta)\right|,
\end{align*}
where the supreme is taken over all $h\in C(\partial\DD)$ with $\|h\|_{C(\partial\DD)}=1$. Applying Proposition \ref{Poi-property-1}(v) with $X=C(\partial\DD)$, we have
\begin{align*}
\|d\nu\|_{{\frak B_{\lambda}}}
&=\sup\lim_{r\rightarrow1-}c_{\lambda}\left|\int_{-\pi}^{\pi}P(h,re^{i\theta})|\sin\theta|^{2\lambda}d\nu(\theta)\right|\\
&=\sup\lim_{r\rightarrow1-}\left|\int_{-\pi}^{\pi}h(\varphi)P(d\nu,re^{i\varphi})dm_{\lambda}(\varphi)\right|,
\end{align*}
which is dominated by $\sup\|h\|_{C(\partial\DD)}C_1(u)=C_1(u)$. This, in conjunction with Proposition \ref{Poi-property-1}(vi), shows that $\|d\nu\|_{{\frak B_{\lambda}}}=C_1(u)$.

\begin{lemma} \label{Poisson-kernel-estimate-a} {\rm (\cite[Corollary 4.3]{LL1})}
Let $z=re^{i\theta}$, $w=e^{i\varphi}$. For
$\theta,\varphi\in[-\pi,\pi]$ and $r\in[0,1)$, we have
\begin{eqnarray*}
P(re^{i\theta},e^{i\varphi})
\lesssim\frac{(1-r)(1-r+|\sin(\theta-\varphi)/2|)^{-2}}{(1-r+|\sin\theta|+|\sin(\theta-\varphi)/2|)^{2\lambda}}
\ln\Big(\frac{|\sin(\theta-\varphi)/2|}{1-r}+2\Big).
\end{eqnarray*}
\end{lemma}

\begin{theorem} \label{harmonic-point-evaluation-a}
Let $u(x,y)$ be a $\lambda$-harmonic function in $\DD$ satisfying (\ref{harmonic-Lp-condition-1}) for $1\le p<\infty$.
Then
\begin{eqnarray}\label{harmonic-point-evaluation-1}
|u(x,y)|\lesssim\frac{(1-|z|)^{-1/p}}{|1-z^2|^{2\lambda/p}}C_p(u)^{1/p},\qquad z=x+iy\in\DD.
\end{eqnarray}
\end{theorem}

\begin{proof}
We first consider the case for $1<p<\infty$. By Lemma \ref{Poi-charact}(i), $u(x,y)$ is the $\lambda$-Poisson
integral $P(f;z)$ with $z=x+iy$ of some $f\in L_{\lambda}^{p}(\partial\DD)$. From (\ref{Poisson-integral-1}), H\"older's inequality gives
\begin{eqnarray} \label{harmonic-point-evaluation-2}
|u(x,y)|\le\|f\|_{L_{\lambda}^{p}(\partial\DD)}\|P(z,\cdot)\|_{L_{\lambda}^{p'}(\partial\DD)},\qquad z=x+iy\in\DD.
\end{eqnarray}
On account of the fact $\|P(z,\cdot)\|_{L_{\lambda}^{1}(\partial\DD)}=1$, one has
\begin{align}\label{Poisson-kernel-norm-1}
\|P(z,\cdot)\|_{L_{\lambda}^{p'}(\partial\DD)}\le
\sup_{\varphi}P(z,e^{i\varphi})^{1/p},\qquad z=re^{i\theta}\in\DD,
\end{align}
but by Lemma \ref{Poisson-kernel-estimate-a},
\begin{eqnarray}\label{Poisson-kernel-estimate-2}
P(re^{i\theta},e^{i\varphi})
\lesssim\frac{(1-r)^{-1}}{(1-r+|\sin\theta|)^{2\lambda}},\qquad \theta,\varphi\in[-\pi,\pi],\,\,r\in[0,1).
\end{eqnarray}
Combining (\ref{harmonic-point-evaluation-2}), (\ref{Poisson-kernel-norm-1}) and (\ref{Poisson-kernel-estimate-2}) yields
\begin{eqnarray*}
|u(x,y)|\lesssim\frac{(1-r)^{-1/p}}{(1-r+|\sin\theta|)^{2\lambda/p}}\|f\|_{L_{\lambda}^{p}(\partial\DD)},\qquad z=x+iy=re^{i\theta}\in\DD,
\end{eqnarray*}
that is identical with (\ref{harmonic-point-evaluation-1}), in view of the fact $|1-z^2|\asymp 1-r+|\sin\theta|$ for $z=re^{i\theta}\in\DD$.

For $p=1$, Lemma \ref{Poi-charact}(ii) implies existence of  a measure $d\nu\in {\frak B}_{\lambda}(\partial\DD)$ so that $u(x,y)$ is its $\lambda$-Poisson
integral $P(d\nu;z)$ with $z=x+iy$. It follows from (\ref{Poisson-Borel-1}) that
$|u(x,y)|\le\|d\nu\|_{{\frak B_{\lambda}}}\sup_{\varphi}P(z,e^{i\varphi})$,
and then, appealing to (\ref{Poisson-kernel-estimate-2}) proves (\ref{harmonic-point-evaluation-1}) for $p=1$.
\end{proof}

We now come to the point estimates of functions in $H_{\lambda}^p(\DD)$ and $A_{\lambda}^p(\DD)$.

\begin{theorem} \label{Hardy-point-evaluation-a}
Let $p\ge p_0$ and $f\in H_{\lambda}^p(\DD)$. Then
\begin{eqnarray}\label{Hardy-point-evaluation-1}
|f(z)|\lesssim\frac{(1-|z|)^{-1/p}}{|1-z^2|^{2\lambda/p}}\|f\|_{H_{\lambda}^{p}},\qquad z\in\DD.
\end{eqnarray}
\end{theorem}

\begin{proof}
For $p\ge p_0$ and $f\in H_{\lambda}^p(\DD)$, let $u(x,y)$ be the $\lambda$-harmonic majorization of $f$ given in Theorem \ref{majorization-a}. Thus
\begin{eqnarray}\label{har-majorization-3}
|f(z)|^{p_0}\leq u(x,y),\qquad z=x+iy\in\DD,
\end{eqnarray}
where for $p >p_0$, $u(x,y)$ is the $\lambda$-Poisson
integral $P(g;z)$ with $z=x+iy$ of some nonnegative function $g\in L_{\lambda}^{p/p_0}(\partial\DD)$ satisfying $\|g\|_{L_{\lambda}^{p/p_0}(\partial\DD)}\asymp\|f\|^{p_0}_{H^{p}_{\lambda}}$,
and for $p=p_0$, the $\lambda$-Poisson integral $P(d\nu;z)$ with $z=x+iy$ of some nonnegative measure $d\nu\in {\frak B}_{\lambda}(\partial\DD)$ satisfying $\|d\nu\|_{{\frak B}_{\lambda}(\partial\DD)}\asymp\|f\|^{p_0}_{H^{p_0}_{\lambda}}$.
Obviously by Lemma \ref{Poi-charact}(iii),
\begin{align*}
C_{p/p_0}(u)=\|g\|^{p/p_0}_{L_{\lambda}^{p/p_0}(\partial\DD)}\asymp\|f\|^{p}_{H^{p}_{\lambda}}<\infty
\end{align*}
for $p>p_0$, and $C_1(u)=\|d\nu\|_{{\frak B_{\lambda}(\partial\DD)}}\asymp\|f\|^{p_0}_{H^{p_0}_{\lambda}}<\infty$ for $p=p_0$.
Now by Theorem \ref{harmonic-point-evaluation-a},
\begin{eqnarray*}
|u(x,y)|\lesssim\frac{(1-|z|)^{-p_0/p}}{|1-z^2|^{2\lambda p_0/p}}\|f\|^{p_0}_{H^{p}_{\lambda}},\qquad z=x+iy\in\DD.
\end{eqnarray*}
Combining this with (\ref{har-majorization-3}) proves (\ref{Hardy-point-evaluation-1}).
\end{proof}

\begin{theorem} \label{Bergman-point-evaluation-a}
Let $p\ge p_0$ and $f\in A_{\lambda}^p(\DD)$. Then
\begin{eqnarray}\label{Bergman-point-evaluation-1}
|f(z)|\lesssim\frac{(1-|z|)^{-2/p}}{|1-z^2|^{2\lambda/p}}\|f\|_{A_{\lambda}^{p}},\qquad z\in\DD.
\end{eqnarray}
\end{theorem}

\begin{proof}
For $0<\rho<1$ set $f_{\rho}(z)=f(\rho z)$. We apply Theorem \ref{Hardy-point-evaluation-a} to $f_{\rho}$ to get
\begin{eqnarray*}
|f(\rho se^{i\theta})|\lesssim\frac{(1-s)^{-1/p}}{|1-s^2e^{2i\theta}|^{2\lambda/p}}\|f_{\rho}\|_{H_{\lambda}^{p}},\qquad \rho,s\in[0,1),
\end{eqnarray*}
but from (\ref{dilation-integral-mean-1}) and Proposition \ref{Bergman-mean-estimate-a},
$$
\|f_{\rho}\|_{H_{\lambda}^{p}}\lesssim M_p(f;\rho)\lesssim(1-\rho)^{-1/p}\|f\|_{A_{\lambda}^{p}}.
$$
Combining the above two estimates and letting $\rho=s=r^{1/2}$ for $r\in[0,1)$, we have
\begin{eqnarray*}
|f(z)|\lesssim\frac{(1-r^{1/2})^{-2/p}}{|1-re^{2i\theta}|^{2\lambda/p}}\|f\|_{A_{\lambda}^{p}},\qquad z=re^{i\theta}\in\DD,
\end{eqnarray*}
that is identical with (\ref{Bergman-point-evaluation-1}), in view of the fact $|1-z^2|\asymp 1-r+|\sin\theta|\asymp|1-re^{2i\theta}|$ for $z=re^{i\theta}\in\DD$.
\end{proof}

For the point estimates of functions in the usual Hardy spaces, see \cite[p. 36]{Du2}, and in the usual Bergman spaces, see \cite{Vu1} and also \cite[p. 79]{DS1}.

\section{Density, completeness, and duality of the $\lambda$-Hardy and $\lambda$-Bergman spaces}

\subsection{Density of $\lambda$-analytic polynomials}

For a $\lambda$-analytic function $f$ on $\DD$, we denote by $S_n:=S^f_n(z)$ the $n$th partial sum of the series (\ref{anal-series-1}).

\begin{lemma}\label{partial-sum-converge-a}
If $f(z)$ is $\lambda$-analytic in $\DD$, then for fixed $0<r_0<1$, $S_n(z)$ converges uniformly to $f(z)$ on $\overline{\DD}_{r_0}$ as $n\rightarrow\infty$, where $\DD_{r_0}=\{z:\, |z|<r_0\}$.
\end{lemma}

Indeed, by (\ref{phi-bound-1}) and Proposition \ref{anal-thm}(iii), $\sup_{z\in\overline{\DD}_{r_0}}|S_n(z)-f(z)|$ is controlled by a multiple of
$\sum_{k=n+1}^{\infty}n^{\lambda}|c_n|r_0^n$, that tends to zero as $n\rightarrow\infty$.

\begin{theorem}\label{Hardy-density-a}
Assume that $p_0<p<\infty$. Then the set of $\lambda$-analytic polynomials is dense in the $\lambda$-Hardy space $H_{\lambda}^p(\DD)$; and in particular, the system $\{\phi_n(z): \,n\in\NN_0\}$
is an orthonormal basis of the Hilbert space $H_{\lambda}^2(\DD)$.
\end{theorem}

For $f\in H_{\lambda}^p(\DD)$ and $s\in(0,1)$, let $f_s(z)=f(sz)$. By Theorem \ref{Hardy-boundary-value-a} $\|f_s-f\|_{H^{p}_{\lambda}}\rightarrow0$ as $s\rightarrow1-$, and by Lemma \ref{partial-sum-converge-a}, for fixed $s\in(0,1)$ $\|(S_n)_s-f_s\|_{H^{p}_{\lambda}}\rightarrow0$ as $n\rightarrow\infty$. Thus the density in Theorem \ref{Hardy-density-a} is concluded.

\begin{theorem} \label{Bergman-density-a}
Assume that $p_0<p<\infty$. Then the set of $\lambda$-analytic polynomials is dense in the $\lambda$-Bergman space $A^{p}_{\lambda}(\DD)$; and in particular, the system
$\big\{\big(\frac{n+\lambda+1}{\lambda+1}\big)^{1/2}\phi_{n}^{\lambda}(z): \,\, n\in\NN_0\big\}$
is an orthonormal basis of the Hilbert space $A^{2}_{\lambda}(\DD)$.
\end{theorem}

\begin{proof}
For $f\in A^{p}_{\lambda}(\DD)$, set $f_s(z)=f(sz)$ for $0\le s<1$.
We assert that $\|f_s-f\|_{A_{\lambda}^{p}}\rightarrow0$ as $s\rightarrow1-$, namely,
\begin{eqnarray}\label{Ap-norm-convergence-1}
\lim_{s\rightarrow1-}\int_0^1M_p(f_s-f;r)^p\,r^{2\lambda+1}dr=0.
\end{eqnarray}
Indeed, applying (\ref{Hp-norm-convergence-1}) yields $\lim_{s\rightarrow1-}M_p(f_s-f;r)=0$ for $0\le r<1$, and for all $0\le r,s<1$, $M_p(f_s-f;r)\lesssim M_p(f;rs)+M_p(f;r)\lesssim M_p(f;r)$ by Lemma \ref{monotonicity-integral-mean-2-a}. Thus (\ref{Ap-norm-convergence-1}) follows immediately by Lebesgue's dominated convergence theorem. In addition, by Lemma \ref{partial-sum-converge-a}, for fixed $s\in(0,1)$ $\|(S_n)_s-f_s\|_{A^{p}_{\lambda}}\rightarrow0$ as $n\rightarrow\infty$. Thus the density in Theorem \ref{Bergman-density-a} is proved.
The last assertion in the theorem follows from (\ref{Bergman-normalization-1}).
\end{proof}

\subsection{Completeness of the $\lambda$-Hardy and $\lambda$-Bergman spaces}

\begin{lemma}\label{analytic-converge-a}
Let $\{f_n\}$ be a sequence of $\lambda$-analytic functions on $\DD$. If $\{f_n\}$ converges uniformly on each compact subset of $\DD$, then its limit function is also $\lambda$-analytic in $\DD$.
\end{lemma}

\begin{proof}
Obviously, the assumption implies that the limit function $f$ of the sequence $\{f_n\}$ exists and is continuous in $\DD$. Furthermore,
by Corollary \ref{reproducing-Bergman-b}, for each $n$ and $r\in(0,1)$ one has
\begin{eqnarray}\label{reproducing-Bergman-3}
f_n(z)=r^{-2\lambda-2}\int_{\DD_r}f_n(w)K_{\lambda}(z/r,w/r)d\sigma_{\lambda}(w),\qquad z\in\DD_r.
\end{eqnarray}
Since $|K_{\lambda}(z/r,w/r)|\lesssim\left(1-|z|/r\right)^{-2\lambda-2}$ for $z,w\in\DD_r$ by Lemma \ref{Bergman-kernel-a}, letting $n\rightarrow\infty$ in (\ref{reproducing-Bergman-3}) asserts that (\ref{reproducing-Bergman-2}) is true for the limit function $f$ and for each $r\in(0,1)$, and so, again by Corollary \ref{reproducing-Bergman-b}, $f$ is $\lambda$-analytic in $\DD$.
\end{proof}

\begin{theorem}\label{completeness-Hardy-a}
For $p_0\le p\le\infty$, the space $H^p_{\lambda}(\DD)$ is complete.
\end{theorem}

\begin{proof}
Let $\{f_n\}$ be a Cauchy sequence in $H^p_{\lambda}(\DD)$ for $p_0\le p<\infty$. For $r\in(0,1)$, it follows from Theorem \ref{Hardy-point-evaluation-a} that
\begin{eqnarray*}
|f_m(z)-f_n(z)|\lesssim(1-r)^{-(2\lambda+1)/p}\|f_m-f_n\|_{H_{\lambda}^{p}},\qquad z\in\overline{\DD}_r,
\end{eqnarray*}
which shows that $\{f_n\}$ converges to a function $f$ uniformly on each compact subset of $\DD$.
Lemma \ref{analytic-converge-a} asserts that $f$ is $\lambda$-analytic in $\DD$. The locally uniform convergence implies, for $r\in[0,1)$,
\begin{align}\label{completeness-Hardy-1}
M_p(f_n-f;r)^p=\lim_{m\rightarrow\infty}\int_{-\pi}^{\pi}|f_n(re^{i\theta})-f_m(re^{i\theta})|^p\,dm_{\lambda}(\theta)
\le\liminf_{m\rightarrow\infty}\|f_n-f_m\|^p_{H_{\lambda}^{p}},
\end{align}
so that $\|f_n-f\|_{H_{\lambda}^{p}}\le\liminf_{m\rightarrow\infty}\|f_n-f_m\|_{H_{\lambda}^{p}}$. Hence $H^p_{\lambda}(\DD)$ for $p_0\le p<\infty$ is complete, and so is $H^{\infty}_{\lambda}(\DD)$.
\end{proof}

\begin{theorem}\label{completeness-Bergman-a}
For $p_0\le p<\infty$, the space $A^p_{\lambda}(\DD)$ is complete.
\end{theorem}

The proof is similar to that of Theorem \ref{completeness-Hardy-a}, by means of Theorem \ref{Bergman-point-evaluation-a} and Lemma \ref{analytic-converge-a}. The only difference is that (\ref{completeness-Hardy-1}) is modified by,
for $r\in(0,1)$,
\begin{align*}
\int_{\DD_r}|f_n(z)-f(z)|^{p}d\sigma_{\lambda}(z)=\lim_{m\rightarrow\infty}\int_{\DD_r}|f_n(z)-f_m(z)|^{p}d\sigma_{\lambda}(z)
\le\liminf_{m\rightarrow\infty}\|f_n-f_m\|^p_{A_{\lambda}^{p}}.
\end{align*}

%\begin{proof}
%Let $\{f_n\}$ be a Cauchy sequence in $A^p_{\lambda}(\DD)$ for $1\le p<\infty$. By Theorem \ref{Bergman-point-evaluation-a},
%\begin{eqnarray*}
%|f_m(z)-f_n(z)|\lesssim(1-r)^{-2(\lambda+1)/p}\|f_m-f_n\|_{A_{\lambda}^{p}},\qquad z\in\overline{\DD}_r,
%\end{eqnarray*}
%which shows that $\{f_n\}$ converges to a function $f$ uniformly on each compact subset of $\DD$.
%Lemma \ref{analytic-converge-a} asserts that $f$ is $\lambda$-analytic in $\DD$. The locally uniform convergence implies
%\begin{align*}
%\int_{\DD_r}|f_n(z)-f(z)|^{p}d\sigma_{\lambda}(z)=\lim_{m\rightarrow\infty}\int_{\DD_r}|f_n(z)-f_m(z)|^{p}d\sigma_{\lambda}(z)
%\le\liminf_{m\rightarrow\infty}\|f_n-f_m\|^p_{A_{\lambda}^{p}},
%\end{align*}
%so that $\|f_n-f\|_{A_{\lambda}^{p}}\le\liminf_{m\rightarrow\infty}\|f_n-f_m\|_{A_{\lambda}^{p}}$. Hence $A^p_{\lambda}(\DD)$ for $1\le p<\infty$ is complete.
%\end{proof}

\subsection{A Szeg\"o type transform}

In order to determine the dual of the $\lambda$-Hardy spaces subsequently, we introduce a Szeg\"o type transform.
Since for $1\le p\le\infty$, the boundary value $f(e^{i\theta})$ of $f\in H_{\lambda}^p(\DD)$ satisfies (\ref{coefficient-2-2}), it follows that
\begin{align}\label{orthogonality-2-1}
\int_{-\pi}^{\pi}f(e^{i\varphi})e^{i\varphi}C(e^{i\varphi},z)dm_{\lambda}(\varphi)=0,\qquad z\in\DD;
\end{align}
and from (\ref{Poisson-kernel-2-1}) and (\ref{Poisson-1}), one has
\begin{align} \label{Cauchy-2-1}
f(z)=\int_{-\pi}^{\pi}f(e^{i\varphi})C(z,e^{i\varphi})
dm_{\lambda}(\varphi),\qquad z\in\DD.
\end{align}
We call the equality (\ref{Cauchy-2-1}) the $\lambda$-Cauchy integral formula of $f\in H_{\lambda}^p(\DD)$ ($1\le p\le\infty$).

In general, a Szeg\"o type transform can be defined, for $h\in L^{1}_{\lambda}(\partial\DD)$, by
\begin{align} \label{Szego-2-1}
S_{\lambda}h(z)=\int_{-\pi}^{\pi}h(e^{i\varphi})C(z,e^{i\varphi})
dm_{\lambda}(\varphi),\qquad z\in\DD.
\end{align}
The operator $S_{\lambda}$ is called the $\lambda$-Szeg\"o transform.

\begin{proposition}\label{Szego-boundedness-a}
For $1<p<\infty$, the $\lambda$-Szeg\"o transform maps $L^{p}_{\lambda}(\partial\DD)$ boundedly onto $H^p_{\lambda}(\DD)$.
\end{proposition}

\begin{proof}
Recall that, for $h\in L^{1}_{\lambda}(\partial\DD)$, its $\lambda$-conjugate Poisson
integral is (cf. \cite[Subsection 5.1]{LL1})
\begin{eqnarray*} \label{conjugate-Poisson-1}
Q(h;z)=\int_{-\pi}^{\pi}h(e^{i\varphi})Q(z,e^{i\varphi})
dm_{\lambda}(\varphi),\qquad z\in\DD,
\end{eqnarray*}
where $Q(z,w)$ is the $\lambda$-conjugate Poisson kernel given by
\begin{align*}
Q(z,w)=-i\left[2C(z,w)-P(z,w)-1-Q_1(z,w)\right]
\end{align*}
with
\begin{eqnarray*}\label{conjugate-ker-3}
Q_1(z,w)=\sum_{n=1}^{\infty}\frac{2\lambda}{\sqrt{n(n+2\lambda)}}\phi_{n}(z)w\phi_{n-1}^{\lambda}(w).
\end{eqnarray*}
It then follows that
\begin{align*}
S_{\lambda}h(z)-\frac{i}{2}Q(h;z)=\frac{1}{2}\int_{-\pi}^{\pi}h(e^{i\varphi})\left[P(z,e^{i\varphi})+1+Q_1(z,e^{i\varphi})\right]
dm_{\lambda}(\varphi),\qquad z\in\DD.
\end{align*}
By \cite[Lemma 5.3]{LL1}, the function $(\theta,\varphi)\mapsto Q_1(re^{i\theta},e^{i\varphi})$ is an integrable kernel in $L^{1}_{\lambda}(\partial\DD)$ uniformly with respect to $r\in[0,1)$, which together with Proposition \ref{Poi-property-1}(iv) yields, for $1\le p\le\infty$,
$$
M_p\left(S_{\lambda}h-2^{-1}iQ(h;\cdot);r\right)\lesssim\|h\|_{L_{\lambda}^{p}(\partial\DD)},\qquad h\in L^{p}_{\lambda}(\partial\DD),\,\,0\le r<1.
$$
Thus for $1<p<\infty$, by the M. Riesz type theorem for $\lambda$-conjugate Poisson
integrals (cf. \cite[Theorem 5.1]{LL1}) we get $M_p(S_{\lambda}h;r)\lesssim\|h\|_{L_{\lambda}^{p}(\partial\DD)}$ for all $h\in L^{p}_{\lambda}(\partial\DD)$ and $0\le r<1$. This finishes the proof of the proposition.
\end{proof}

\begin{corollary}\label{Szego-boundedness-b}
If $h\in L^{p}_{\lambda}(\partial\DD)$ for $1<p<\infty$, then
\begin{align} \label{orthogonality-2-2}
\int_{-\pi}^{\pi}\left(h(e^{i\theta})-S_{\lambda}h(e^{i\theta})\right)\overline{f(e^{i\theta})}
dm_{\lambda}(\theta)=0
\end{align}
for all $f\in H^{p'}_{\lambda}(\DD)$.
\end{corollary}

\begin{proof}
From (\ref{Cauchy-kernel-2-1}), (\ref{Poisson-kernel-2-1}), (\ref{Poisson-integral-1}) and (\ref{Szego-2-1}), we have
\begin{align*}
P(h;z)-S_{\lambda}h(z)=\bar{z}\int_{-\pi}^{\pi}h(e^{i\varphi})e^{i\varphi} C(e^{i\varphi},z)
dm_{\lambda}(\varphi),\qquad z\in\DD,
\end{align*}
and for $f\in H^{p'}_{\lambda}(\DD)$, applying Fubini's theorem and (\ref{orthogonality-2-1}) gives
\begin{align} \label{orthogonality-2-3}
&\int_{-\pi}^{\pi}\left(P(h;re^{i\theta})-S_{\lambda}h(re^{i\theta})\right)\overline{f(e^{i\theta})}
dm_{\lambda}(\theta)\nonumber\\
&=\int_{-\pi}^{\pi}h(e^{i\varphi})re^{i\varphi}\,\overline{\int_{-\pi}^{\pi} e^{i\theta}C(e^{i\theta},re^{i\varphi})f(e^{i\theta})
dm_{\lambda}(\theta)}\,dm_{\lambda}(\varphi)\nonumber\\
&=0.
\end{align}
Since $P(h;re^{i\theta})-S_{\lambda}h(re^{i\theta})$ converges to $h(e^{i\theta})-S_{\lambda}h(e^{i\theta})$ in the $L^{p}_{\lambda}(\partial\DD)$-norm by Propositions \ref{Poi-property-1}(v) and \ref{Szego-boundedness-a}, and then (\ref{Hp-norm-convergence-1}), letting $r\rightarrow1-$ in (\ref{orthogonality-2-3}) proves (\ref{orthogonality-2-2}).
\end{proof}

\subsection{Duality for the $\lambda$-Hardy spaces and the $\lambda$-Bergman spaces}

Let $X$ be a Banach space, and let $S$ be a (closed) subspace of $X$. The annihilator of $S$, denoted by $S^{\perp}$, is the set of all linear functionals
$L\in X^*$ such that $L(x)=0$ for all $x\in S$. A consequence of the Hahn-Banach theorem reads as follows (cf. \cite[Theorem 7.1]{Du2}).

\begin{lemma}\label{Duality-Banach-space-a}
The dual $S^*$ of $S$ is isometrically isomorphic to the quotient space $X^*/S^{\perp}$. Furthermore, for each fixed $L\in X^*$,
\begin{align*}
\sup_{x\in S,\|x\|\le1}|L(x)|=\min_{\widetilde{L}\in S^{\perp}}\|L+\widetilde{L}\|,
\end{align*}
where ``$\min$" indicates that the infimum is attained.
\end{lemma}

Theorem \ref{Hp-Lp-thm-1} implies that, for $1\le p\le\infty$, the set of boundary functions of
$H^p_{\lambda}(\DD)$, namely, $H^p_{\lambda}(\partial\DD):=\{f(e^{i\theta}):\,\,f\in H^p_{\lambda}(\DD)\}$, consists of those elements in $L^{p}_{\lambda}(\partial\DD)$ satisfying the condition (\ref{coefficient-2-2}), and is isometrically isomorphic to $H^p_{\lambda}(\DD)$.

For $1\le p\le\infty$, define $\left(zH^p_{\lambda}\right)(\DD)=\left\{zf(z):\,\, f\in H^p_{\lambda}(\DD)\right\}$, and similarly,
$$
\left(e^{i\theta}H^p_{\lambda}\right)(\partial\DD)=\left\{e^{i\theta}f(e^{i\theta}):\,\, f\in H^p_{\lambda}(\DD)\right\}.
$$

\begin{theorem}\label{Duality-Hardy-space-a}
For $1\le p<\infty$, the dual space $H^p_{\lambda}(\partial\DD)^*$ (or $H^p_{\lambda}(\DD)^*$) is isometrically isomorphic to the quotient space $L^{p'}_{\lambda}(\partial\DD)/(e^{i\theta}H^{p'}_{\lambda})(\partial\DD)$, where $p^{-1}+p'^{-1}=1$. Furthermore if $1<p<\infty$, then $H^p_{\lambda}(\partial\DD)^*$ is isomorphic to $H^{p'}_{\lambda}(\partial\DD)$ in the sense that, each $L\in H^p_{\lambda}(\partial\DD)^*$ can be represented by
\begin{align}\label{Hp-functional-1}
L(f)=\int_{-\pi}^{\pi}f(e^{i\theta})\overline{g(e^{i\theta})}
dm_{\lambda}(\theta),\qquad f\in H^p_{\lambda}(\partial\DD),
\end{align}
with a unique function $g\in H^{p'}_{\lambda}(\partial\DD)$ satisfying
$C_p\|g\|_{H_{\lambda}^{p'}}\le\|L\|\le\|g\|_{H_{\lambda}^{p'}}$, where the constant $C_p$ is independent of $g$. While each $L\in H^1_{\lambda}(\partial\DD)^*$ can be represented by (\ref{Hp-functional-1}) with some $g\in L^{\infty}(\partial\DD)$ satisfying
$\|g\|_{L^{\infty}(\partial\DD)}=\|L\|$.
\end{theorem}

\begin{proof}
The proof of the theorem is in the same style as that of \cite[Theorem 7.3]{Du2}.
The first part is a consequence of Lemma \ref{Duality-Banach-space-a}, as long as the equality $\left(H_{\lambda}^p(\DD)\right)^{\perp}=(e^{i\theta}H^{p'}_{\lambda})(\partial\DD)$ for $1\le p<\infty$ is shown.
Indeed, the density of $\lambda$-analytic polynomials in $H_{\lambda}^p(\DD)$, by Theorem \ref{Hardy-density-a}, implies that $g\in L^{p'}_{\lambda}(\partial\DD)$ annihilates $H^p_{\lambda}(\partial\DD)$ if and only if
\begin{align}\label{coefficient-5-1}
\int_{-\pi}^{\pi}\phi_{n}(e^{i\theta})g(e^{i\theta})dm_{\lambda}(\theta)=0,\qquad n\in\NN_0.
\end{align}
Note that (\ref{coefficient-5-1}) is equivalent to say that $e^{-i\theta}g(e^{i\theta})\in L^{p'}_{\lambda}(\partial\DD)$ satisfies (\ref{coefficient-2-2}), namely, $e^{-i\theta}g(e^{i\theta})\in H^{p'}_{\lambda}(\partial\DD)$, and hence $g\in(e^{i\theta}H^{p'}_{\lambda})(\partial\DD)$.

To show the second part of the theorem, that is, each $L\in H^p_{\lambda}(\partial\DD)^*$ for $1<p<\infty$, has the representation (\ref{Hp-functional-1}), we use the $\lambda$-Szeg\"o transform $S_{\lambda}$ defined in the last subsection.
Since, by the Hahn-Banach theorem, $L$ can be extended to a bounded linear functional on $L^p_{\lambda}(\partial\DD)$ with the same norm, the Riesz representation theorem yields a function $h\in L^{p'}_{\lambda}(\partial\DD)$, with $\|h\|_{L_{\lambda}^{p'}(\partial\DD)}=\|L\|$, satisfying
\begin{align}\label{Lp-functional-1}
L(f)=\int_{-\pi}^{\pi}f(e^{i\theta})h(e^{i\theta})dm_{\lambda}(\theta)
\end{align}
for all $f\in L^p_{\lambda}(\partial\DD)$. If we put $g(z)=S_{\lambda}\bar{h}(z)$, then by Proposition \ref{Szego-boundedness-a}, $g(z)\in H^{p'}_{\lambda}(\DD)$ and hence $g(e^{i\theta})\in H^{p'}_{\lambda}(\partial\DD)$. Furthermore
$\|g\|_{H_{\lambda}^{p'}}\le C_p^{-1}\|h\|_{L_{\lambda}^{p'}(\partial\DD)}=C_p^{-1}\|L\|$, and by Corollary \ref{Szego-boundedness-b} and (\ref{Lp-functional-1}), the representation (\ref{Hp-functional-1}) is concluded.
If $g(z)=\sum_{n=0}^{\infty}a_{n}\phi_{n}(z)$, then by (\ref{coefficient-2-3}),
$L(\phi_n)=\overline{a_n}$ for $n\in\NN_0$. This shows that $g\in H^{p'}_{\lambda}(\DD)$ is uniquely determined by $L$.
On the other hand, every $g\in H^{p'}_{\lambda}(\DD)$ defines a functional $L\in H^p_{\lambda}(\partial\DD)^*$ by (\ref{Hp-functional-1}), certainly satisfying $\|L\|\le\|g\|_{H_{\lambda}^{p'}}$.
The assertion for $p=1$ is obvious. The proof of the theorem is completed.
\end{proof}

Finally we consider the dual of the $\lambda$-Bergman spaces $A^{p}_{\lambda}(\DD)$ for $1<p<\infty$. The duality theorem of the usual Bergman spaces $A^p$ with $1<p<\infty$ was proved in \cite{ZY} (see also \cite[p. 35]{DS1}).

\begin{theorem}\label{Duality-Bergman-space-a}
For $1<p<\infty$, The dual space $A^{p}_{\lambda}(\DD)^*$ is isomorphic to $A^{p'}_{\lambda}(\DD)$ in the sense that, each $L\in A^p_{\lambda}(\DD)^*$ can be represented by
\begin{align*}
L(f)=\int_{\DD}f(z)\overline{g(z)}d\sigma_{\lambda}(z),\qquad f\in A^p_{\lambda}(\DD),
\end{align*}
with a unique function $g\in A^{p'}_{\lambda}(\DD)$ satisfying
$C_p\|g\|_{A_{\lambda}^{p'}}\le\|L\|\le\|g\|_{A_{\lambda}^{p'}}$, where the constant $C_p$ is independent of $g$.
\end{theorem}

The proof of the theorem is nearly an adaptation of that of its analog in the classical case given in \cite[p. 37]{DS1}). On the one hand, we extend $L\in A^p_{\lambda}(\DD)^*$, by the Hahn-Banach theorem, to a bounded linear functional on $L^p_{\lambda}(\DD)$ with the same norm, and the Riesz representation theorem asserts that there exists an $h\in L^{p'}_{\lambda}(\DD)$ satisfying $\|h\|_{L_{\lambda}^{p'}(\DD)}=\|L\|$, so that
$L(f)=\int_{\DD}f(z)h(z)d\sigma_{\lambda}(z)$ for all $f\in L^p_{\lambda}(\DD)$. If we put $g(z)=(P_{\lambda}\bar{h})(z)$, then by Theorem \ref{Bergman-projection-a}, $g\in A^{p'}_{\lambda}(\DD)$ and
$\|g\|_{A_{\lambda}^{p'}}\le C_p^{-1}\|h\|_{L_{\lambda}^{p'}(\DD)}=C_p^{-1}\|L\|$. For $f\in A^{p}_{\lambda}(\DD)$, by Proposition \ref{reproducing-Bergman-a} one has
$L(f)=\int_{\DD}\int_{\DD}f(w)K_{\lambda}(z,w)d\sigma_{\lambda}(w)\,h(z)d\sigma_{\lambda}(z)$, and then, by Corollary \ref{Bergman-projection-a-1}, Fubini's theorem gives
\begin{align}\label{Ap-functional-3}
L(f)=\int_{\DD}f(w)\overline{g(w)}\,d\sigma_{\lambda}(w).
\end{align}
If $g(z)=\sum_{n=0}^{\infty}a_{n}\phi_{n}(z)$, then by (\ref{coefficient-2-1}),
$L(\phi_n)=\frac{\lambda+1}{n+\lambda+1}\overline{a_n}$ for $n\in\NN_0$. This shows that $g\in A^{p'}_{\lambda}(\DD)$ is uniquely determined by $L$.
On the other hand, every $g\in A^{p'}_{\lambda}(\DD)$ defines a functional $L\in H^p_{\lambda}(\partial\DD)^*$ by (\ref{Ap-functional-3}), satisfying $\|L\|\le\|g\|_{A_{\lambda}^{p'}}$.

\section{Characterization and interpolation of the $\lambda$-Bergman spaces}

\subsection{Boundedness of some operators}

We consider the following two operators
\begin{align}
(P_{\lambda,1}f)(z)&=
\int_{\DD}f(w)K_{\lambda,1}(z,w)(1-|w|^2)d\sigma_{\lambda}(w),\qquad z\in\DD,\label{Bergman-projection-2}\\
(P_{\lambda,2}f)(z)&=
\int_{\DD}f(w)K_{\lambda,2}(z,w)(1-|w|^2)^2d\sigma_{\lambda}(w),\qquad z\in\DD,\label{Bergman-projection-3}
\end{align}
where
\begin{align}
K_{\lambda,1}(z,w)&=\sum_{n=0}^{\infty}\frac{(n+\lambda+2)(n+\lambda+1)}{\lambda+1}
\phi_{n}(z)\overline{\phi_{n}(w)},\label{Bergman-kernel-5}\\
K_{\lambda,2}(z,w)&=\frac{1}{2\lambda+2}\sum_{n=0}^{\infty}
\frac{\Gamma(n+\lambda+4)}{\Gamma(n+\lambda+1)}
\phi_{n}(z)\overline{\phi_{n}(w)}.\label{Bergman-kernel-6}
\end{align}

Similarly to the $\lambda$-Bergman kernel $K_{\lambda}(z,w)$, it follows from (\ref{phi-bound-1}) and (\ref{phi-bound-2}) that the series in (\ref{Bergman-kernel-5}) and (\ref{Bergman-kernel-6}) are convergent absolutely for $zw\in\DD$ and uniformly for $zw$ in a compact subset of $\DD$; and for fixed $w\in\overline{\DD}$ the function $z\mapsto K_{\lambda,j}(z,w)$ ($j=1,2$) is $\lambda$-analytic in $\DD$.

Moreover we also have

\begin{proposition} \label{reproducing-Bergman-c}
For $f\in L^{1}(\DD;(1-|z|^2)^{j}d\sigma_{\lambda})$ with $j=1$ or $2$, the integrals on the right hand side of (\ref{Bergman-projection-2}) and (\ref{Bergman-projection-3}) are well defined for $z\in\DD$ and are $\lambda$-analytic in $z\in\DD$; and if $f\in L^{1}(\DD;(1-|z|^2)^{j}d\sigma_{\lambda})$ is $\lambda$-analytic in $\DD$, then $P_{\lambda,j}f=f$ with $j=1$ or $2$.
\end{proposition}

Indeed as in the proof of Proposition \ref{reproducing-Bergman-a}, by (\ref{phi-bound-1}) termwise integration of
$f(w)K_{\lambda,j}(z,w)$ over $\DD$ with respect to the measure $(1-|w|^2)^{j}d\sigma_{\lambda}(w)$ is legitimate, and by Proposition \ref{anal-thm}(iii) the resulting series represents a $\lambda$-analytic function in $\DD$.
If $f\in L^{1}(\DD;(1-|z|^2)^{j}d\sigma_{\lambda})$ is $\lambda$-analytic in $\DD$, by Proposition \ref{anal-thm}(iii) it has the representation $f(z)=\sum_{n=0}^{\infty}c_{n}\phi_{n}(z)$ for $z\in\DD$, and since, by Proposition \ref{orthonirmal-basis-a},
\begin{align*}
\int_{\DD}\left|\phi_{n}(z)\right|^2\,(1-|z|^2)^{j}d\sigma_{\lambda}(z)
=\frac{(\lambda+1)\Gamma(j+1)\Gamma(n+\lambda+1)}{\Gamma(n+\lambda+j+2)},
\end{align*}
it follows that, for $r\in(0,1)$,
$$
\int_{\DD}\overline{\phi_{n}(z)}f(rz)\,(1-|z|^2)^{j}d\sigma_{\lambda}(z)=\frac{(\lambda+1)\Gamma(j+1)\Gamma(n+\lambda+1)}{\Gamma(n+\lambda+j+2)}r^nc_n.
$$
Making change of variables $z\mapsto z/r$, one has
$$
\int_{\DD_r}\overline{\phi_{n}(z)}f(z)\,(r^2-|z|^2)^{j}d\sigma_{\lambda}(z)=\frac{(\lambda+1)\Gamma(j+1)\Gamma(n+\lambda+1)}{\Gamma(n+\lambda+j+2)}r^{2n+2\lambda+2j+2}c_n,
$$
and then, letting $r\rightarrow1-$ yields
\begin{align}\label{Bergman-coefficient-1}
\int_{\DD}\overline{\phi_{n}(z)}f(z)\,(1-|z|^2)^{j}d\sigma_{\lambda}(z)=\frac{(\lambda+1)\Gamma(j+1)\Gamma(n+\lambda+1)}{\Gamma(n+\lambda+j+2)}c_n.
\end{align}
Finally termwise integration of
$f(w)K_{\lambda,j}(z,w)$ over $\DD$ with respect to the measure $(1-|w|^2)^{j}d\sigma_{\lambda}(w)$ proves $P_{\lambda,j}f=f$ with $j=1$ or $2$.

Similarly to Lemma \ref{Bergman-kernel-a}, we have
\begin{lemma}\label{Bergman-kernel-b}
For $j=1,2$ and for $|zw|<1$,
\begin{align}\label{Bergman-kernal-7}
|K_{\lambda,j}(z,w)|\lesssim \frac{(|1-z\overline{w}|+|1-zw|)^{-2\lambda}}{|1-z\overline{w}|}
\left(\frac{1}{|1-z\overline{w}|^{j+1}}+\frac{1}{|1-zw|^{j+1}}\right).
\end{align}
\end{lemma}

On the case with $j=1$, for $|zw|<1$ and $z=re^{i\theta}$ one has
\begin{align}\label{Bergman-kernal-8}
(\lambda+1)K_{\lambda,1}(z,w)=\left[\left(r\frac{d}{dr}\right)^2+(2\lambda+3)r\frac{d}{dr}+(\lambda+2)(\lambda+1)\right]C(z,w).
\end{align}
Since, as indicated in the proof of Lemma \ref{Bergman-kernel-a}, $\left|r\frac{d}{dr}\left[C(z,w)\right]\right|$ has the same upper bound as in (\ref{Bergman-kernal-2}),
on account of Lemma \ref{Cauchy-estimate-a} it suffices to show
\begin{align*}
\left|\left(r\frac{d}{dr}\right)^2\left[C(z,w)\right]\right|\lesssim \frac{(|1-z\overline{w}|+|1-zw|)^{-2\lambda}}{|1-z\overline{w}|}
\left(\frac{1}{|1-z\overline{w}|^{2}}+\frac{1}{|1-zw|^{2}}\right).
\end{align*}
This can be completed along the line of the proof of Lemma \ref{Bergman-kernel-a}, although computations are more complicated. We leave the details to the readers.

Correspondingly, on the case with $j=2$ we have a similar formula to (\ref{Bergman-kernal-8}), and the proof of (\ref{Bergman-kernal-7}) is reduced to show
\begin{align*}
\left|\left(r\frac{d}{dr}\right)^3\left[C(z,w)\right]\right|\lesssim \frac{(|1-z\overline{w}|+|1-zw|)^{-2\lambda}}{|1-z\overline{w}|}
\left(\frac{1}{|1-z\overline{w}|^{3}}+\frac{1}{|1-zw|^{3}}\right).
\end{align*}
We again leave the details to the readers.

We will also need the operator $\widetilde{P}_{\lambda,1}$ and $\widetilde{P}_{\lambda,2}$ defined by
\begin{align}
(\widetilde{P}_{\lambda,1}f)(z)&=\int_{\DD}f(w)\widetilde{K}_{\lambda,1}(z,w)(1-|w|^{2})d\sigma_{\lambda}(w),\label{Bergman-projection-4}\\
(\widetilde{P}_{\lambda,2}f)(z)&=\int_{\DD}f(w)\widetilde{K}_{\lambda,2}(z,w)d\sigma_{\lambda}(w),\label{Bergman-projection-5}
\end{align}
where
\begin{align}
\widetilde{K}_{\lambda,1}(z,w)&=\sum_{n=0}^{\infty}
\frac{(n+\lambda+3)(n+\lambda+2)}{2\lambda+2}\phi_{n}(z)\overline{\phi_{n}(w)},\label{Bergman-kernel-9}\\
\widetilde{K}_{\lambda,2}(z,w)&=K_{\lambda,2}(z,w)(1-|z|^2)(1-|w|^2).\label{Bergman-kernel-10}
\end{align}

Similarly to Lemma \ref{Bergman-kernel-b} with $j=1$, we have
\begin{lemma}\label{Bergman-kernel-c}
For $|zw|<1$,
\begin{align*}
|\widetilde{K}_{\lambda,1}(z,w)|\lesssim \frac{(|1-z\overline{w}|+|1-zw|)^{-2\lambda}}{|1-z\overline{w}|}
\left(\frac{1}{|1-z\overline{w}|^{2}}+\frac{1}{|1-zw|^{2}}\right).
\end{align*}
\end{lemma}

\begin{theorem}\label{Bergman-projection-b}
For $1\le p<\infty$, the operators $P_{\lambda,1}$, $P_{\lambda,2}$ and $\widetilde{P}_{\lambda,1}$ are bounded from $L_{\lambda}^{p}(\DD)$ into $A^{p}_{\lambda}(\DD)$.
\end{theorem}

\begin{proof}
We first show that, for $0<\alpha<j+1$ with $j=1$ or $2$,
\begin{align}\label{Bergman-kernal-12}
I_{j,\alpha}(z):=\int_{\mathbb{D}}|K_{\lambda,j}(z,w)|(1-|w|^{2})^{j-\alpha}d\sigma_{\lambda}(w)\lesssim(1-|z|^2)^{-\alpha}, \qquad z\in\DD.
\end{align}

Similarly to (\ref{Bergman-kernal-3}) and (\ref{Bergman-kernal-4}), it follows from (\ref{Bergman-kernal-7}) that, for $z=re^{i\theta}$, $w=se^{i\varphi}\in\DD$,
\begin{align*}
|K_{\lambda,j}(z,w)|\lesssim \Phi_{j,r,\theta}(s,\varphi)+\Phi_{j,r,\theta}(s,-\varphi),
\end{align*}
where
\begin{align*}
 \Phi_{j,r,\theta}(s,\varphi)
 =\frac{\left(1-rs+|\sin\theta|+|\sin\varphi|\right)^{-2\lambda}}{\left(1-rs+\left|\sin(\theta-\varphi)/2\right|\right)^{j+2}}.
\end{align*}
We have
\begin{align*}
\int_{\mathbb{D}}|K_{\lambda,j}(z,w)|(1-|w|^2)^{j-\alpha}d\sigma_{\lambda}(w)
\lesssim\int_0^1\int_{-\pi}^{\pi}\Phi_{j,r,\theta}(s,\varphi)(1-s)^{j-\alpha}|\sin\varphi|^{2\lambda}d\varphi ds,
\end{align*}
since the contribution of $\Phi_{j,r,\theta}(s,-\varphi)$ to the integral is the same as that of $\Phi_{j,r,\theta}(s,\varphi)$.
Thus
\begin{align*}
\int_{\mathbb{D}}|K_{\lambda,j}(z,w)|(1-|w|^2)^{j-\alpha}d\sigma_{\lambda}(w)
\lesssim\int_0^1\int_{-\pi}^{\pi}\frac{(1-s)^{j-\alpha}}{\left(1-rs+\left|\sin(\theta-\varphi)/2\right|\right)^{j+2}}\,d\varphi ds,
\end{align*}
from which (\ref{Bergman-kernal-12}) is yields after elementary calculations.

For $f\in L_{\lambda}^{1}(\DD)$, we have
\begin{align*}
\int_{\DD}\left|(P_{\lambda,j}f)(z)\right|d\sigma_{\lambda}(z)\lesssim
\int_{\DD}|f(w)|\int_{\DD}|K_{\lambda,j}(z,w)|d\sigma_{\lambda}(z)(1-|w|^2)^{j}d\sigma_{\lambda}(w).
\end{align*}
Applying (\ref{Bergman-kernal-12}) with $\alpha=j$ gives $\|P_{\lambda,j}f\|_{L_{\lambda}^{1}(\DD)}\lesssim\|f\|_{L_{\lambda}^{1}(\DD)}$.

For $1<p<\infty$, take $h(z)=(1-|z|^2)^{-1/pp'}$, where $p^{-1}+p^{'-1}=1$. Applying (\ref{Bergman-kernal-12}), first with $\alpha=p^{-1}$, one has
\begin{align*}
\int_{\mathbb{D}}|K_{\lambda,j}(z,w)|(1-|w|^{2})^{j}h(w)^{p'}d\sigma_{\lambda}(w)\lesssim h(z)^{p'}, \qquad z\in\DD,
\end{align*}
and then, with $\alpha=j+p'^{-1}$,
\begin{align*}
\int_{\mathbb{D}}|K_{\lambda,j}(z,w)|(1-|w|^{2})^{j}h(z)^{p}d\sigma_{\lambda}(z)\lesssim h(w)^{p}, \qquad z\in\DD.
\end{align*}
By Schur's theorem, $\|P_{\lambda,j}f\|_{L_{\lambda}^{p}(\DD)}\lesssim\|f\|_{L_{\lambda}^{p}(\DD)}$ for $j=1$ or $2$. Furthermore,
Proposition \ref{reproducing-Bergman-c} implies that the images of $f\in L_{\lambda}^{p}(\DD)$ ($1\le p<\infty$) under the mappings $P_{\lambda,1}$ and $P_{\lambda,2}$ are $\lambda$-analytic in $\DD$.

The assertion on the operator $\widetilde{P}_{\lambda,1}$ can be proved by the same way above.
\end{proof}

\begin{theorem}\label{Bergman-projection-c}
For $1\le p\le\infty$, the operator $\widetilde{P}_{\lambda,2}$ is bounded from $L_{\lambda}^{p}(\DD)$ into itself.
\end{theorem}

\begin{proof}
We note that $\widetilde{K}_{\lambda,2}(z,w)=\overline{\widetilde{K}_{\lambda,2}(w,z)}$, and by means of (\ref{Bergman-kernal-12}) with $j=2$ and $\alpha=1$,
\begin{align*}
\int_{\DD}|\widetilde{K}_{\lambda,2}(z,w)|d\sigma_{\lambda}(w)\lesssim1.
\end{align*}
This shows that, for $1\le p\le\infty$, the operator $\widetilde{P}_{\lambda,2}$ is bounded from $L^{p}_{\lambda}(\DD)$ into itself.
\end{proof}

\subsection{A characterization of the $\lambda$-Bergman spaces by the operator $D_z$}

Characterizations of the usual Bergman spaces by derivatives can be found in \cite[pp. 56-58]{Zhu1}. Here we make a modification, which is more apt for the $\lambda$-Bergman spaces.

\begin{lemma}\label{derivative-Bergman-a}
For $1\le p\le\infty$, the mapping $f(z)\mapsto(1-|z|^{2})D_{z}\left(zf(z)\right)$ is bounded from $A^{p}_{\lambda}(\DD)$ into $L^{p}_{\lambda}(\DD)$.
\end{lemma}

\begin{proof}
For $f\in A^{p}_{\lambda}(\DD)$, assume that $f(z)=\sum_{n=0}^{\infty}c_{n}\phi_{n}(z)$ ($z\in\DD$) by Proposition \ref{anal-thm}(iii). By (\ref{phi-bound-2}),
taking termwise differentiation $\partial_z$ in $\DD$ to $zf(z)$ is legitimate,
and then, by (\ref{Tzphi-2}),
\begin{align}\label{derivative-analytic-1}
D_{z}\left(zf(z)\right)=\sum_{n=0}^{\infty}c_{n}(n+\lambda+1)\phi_{n}(z),
\end{align}
and from (\ref{Bergman-coefficient-1}),
\begin{align*}
(n+\lambda+1)c_n=\frac{\Gamma(n+\lambda+4)}{(\lambda+1)\Gamma(n+\lambda+1)}\int_{\DD}\overline{\phi_{n}(w)}f(w)\,(1-|w|^2)d\sigma_{\lambda}(w)-2c_n.
\end{align*}
Consequently, in view of (\ref{Bergman-kernel-6}) and (\ref{derivative-analytic-1}) we obtain
\begin{align*}
D_{z}\left(zf(z)\right)=2\int_{\DD}f(w)K_{\lambda,2}(z,w)(1-|w|^2)d\sigma_{\lambda}(w)-2f(z),
\end{align*}
so that, from (\ref{Bergman-projection-5}) and (\ref{Bergman-kernel-10}),
\begin{align*}
(1-|z|^2)D_{z}\left(zf(z)\right)=2(\widetilde{P}_{\lambda,2}f)(z)-2(1-|z|^2)f(z).
\end{align*}
Therefore, by Theorem \ref{Bergman-projection-c},
$\|(1-|z|^2)D_{z}\left(zf(z)\right)\|_{L_{\lambda}^{p}(\DD)}\lesssim\|\widetilde{P}_{\lambda,2}f\|_{L_{\lambda}^{p}(\DD)}+\|f\|_{A_{\lambda}^{p}}\lesssim\|f\|_{A_{\lambda}^{p}}$
for $1\le p\le\infty$.
\end{proof}

\begin{lemma}\label{derivative-Bergman-b}
If $f$ is $\lambda$-analytic in $\DD$ and satisfies $(1-|z|^{2})^2D_{z}\left(zf(z)\right)\in L^{1}_{\lambda}(\DD)$, then
\begin{align}\label{derivative-Bergman-1}
f(z)=\int_{\DD}\widetilde{K}_{\lambda,1}(z,w)D_{w}\left(wf(w)\right)(1-|w|^{2})^2d\sigma_{\lambda}(w),\qquad z\in\DD,
\end{align}
where $\widetilde{K}_{\lambda,1}(z,w)$ is given by (\ref{Bergman-kernel-9}).
Furthermore, if in addition $(1-|z|^{2})D_{z}\left(zf(z)\right)\in L^{p}_{\lambda}(\DD)$ for $1\le p<\infty$, then $f\in A^{p}_{\lambda}(\DD)$
and $\|f\|_{A_{\lambda}^{p}}\lesssim\|(1-|z|^2)D_{z}\left(zf(z)\right)\|_{L_{\lambda}^{p}(\DD)}$.
\end{lemma}

\begin{proof}
Set
$$
g(z)=\int_{\DD}\widetilde{K}_{\lambda,1}(z,w)D_{w}\left(wf(w)\right)(1-|w|^{2})^2d\sigma_{\lambda}(w),\qquad z\in\DD.
$$
For $(1-|z|^{2})^2D_{z}\left(zf(z)\right)\in L^{1}_{\lambda}(\DD)$, in view of (\ref{phi-bound-1}) and (\ref{phi-bound-2}) the function $g$ is well defined and $\lambda$-analytic in $\DD$, and moreover,
\begin{align*}
D_{z}\left(zg(z)\right)=\int_{\DD}D_{z}\left(z\widetilde{K}_{\lambda,1}(z,w)\right)D_{w}\left(wf(w)\right)(1-|w|^{2})^2d\sigma_{\lambda}(w),\qquad z\in\DD.
\end{align*}
In the meanwhile, from (\ref{Tzphi-2}), (\ref{Bergman-kernel-6}) and (\ref{Bergman-kernel-9}),
\begin{align*}
D_{z}\left(z\widetilde{K}_{\lambda,1}(z,w)\right)=K_{\lambda,2}(z,w),
\end{align*}
and hence, in view of (\ref{Bergman-projection-3}),
$$
D_{z}\left(zg(z)\right)=P_{\lambda,2}\left[D_{w}\left(wf(w)\right)\right](z),\qquad z\in\DD.
$$
Since $D_{z}\left(zf(z)\right)\in L^{1}(\DD;(1-|z|^2)^2d\sigma_{\lambda})$, and from (\ref{derivative-analytic-1}), $D_{z}\left(zf(z)\right)$ is $\lambda$-analytic in $\DD$, it follows that $P_{\lambda,2}\left[D_{w}\left(wf(w)\right)\right](z)=D_{z}\left(zf(z)\right)$ by Proposition \ref{reproducing-Bergman-c} with $j=2$. Therefore
$$
D_{z}\left(zg(z)\right)=D_{z}\left(zf(z)\right),\qquad z\in\DD.
$$
In view of the $\lambda$-analyticity of $f$ and $g$, comparing the coefficients of $D_{z}\left(zg(z)\right)$ and $D_{z}\left(zf(z)\right)$ as in (\ref{derivative-analytic-1}) yields that $g=f$ in $\DD$. This proves (\ref{derivative-Bergman-1}).

On account of (\ref{Bergman-projection-4}), the equality (\ref{derivative-Bergman-1}) says $f(z)=(\widetilde{P}_{\lambda,1}F)(z)$ with $F(z)=(1-|z|^{2})D_{z}\left(zf(z)\right)$. Now if $F\in L^{p}_{\lambda}(\DD)$ for $1\le p<\infty$, then Theorem \ref{Bergman-projection-b} tells us that $f\in A^{p}_{\lambda}(\DD)$ and $\|f\|_{A_{\lambda}^{p}}\lesssim\|(1-|z|^2)D_{z}\left(zf(z)\right)\|_{L_{\lambda}^{p}(\DD)}$.
\end{proof}

Collecting Lemmas \ref{derivative-Bergman-a} and \ref{derivative-Bergman-b} we have the following theorem.

\begin{theorem}\label{derivative-Bergman-c}
Suppose that $1\le p<\infty$ and $f$ is $\lambda$-analytic in $\DD$. Then $f\in A^{p}_{\lambda}(\DD)$ if and only if $(1-|z|^{2})D_{z}\left(zf(z)\right)\in L^{p}_{\lambda}(\DD)$. Moreover
$$
\|f\|_{A_{\lambda}^{p}}\asymp\|(1-|z|^2)D_{z}\left(zf(z)\right)\|_{L_{\lambda}^{p}(\DD)}.
$$
\end{theorem}

Similarly to Lemma \ref{derivative-Bergman-b} one can prove the following lemma.

\begin{proposition}\label{derivative-Bergman-d}
If $f\in A^{1}_{\lambda}(\DD)$, then
$$
f(z)=\int_{\DD}\widetilde{K}_{\lambda}(z,w)D_{w}\left(wf(w)\right)(1-|w|^{2})d\sigma_{\lambda}(w),\qquad z\in\DD,
$$
where
\begin{align}\label{Bergman-kernel-15}
\widetilde{K}_{\lambda}(z,w)=\sum_{n=0}^{\infty}
\frac{n+\lambda+2}{\lambda+1}\phi_{n}(z)\overline{\phi_{n}(w)}.
\end{align}
\end{proposition}

\begin{proof}
Set
$$
g(z)=\int_{\DD}\widetilde{K}_{\lambda}(z,w)D_{w}\left(wf(w)\right)(1-|w|^{2})d\sigma_{\lambda}(w),\qquad z\in\DD.
$$
It suffices to show that $g=f$ in $\DD$.

For $f\in A^{1}_{\lambda}(\DD)$, $(1-|z|^{2})D_{z}\left(zf(z)\right)\in L^{1}_{\lambda}(\DD)$ by Lemma \ref{derivative-Bergman-a}, and the function $g$ is well defined in $\DD$. Moreover in view of (\ref{phi-bound-1}) and (\ref{phi-bound-2}), $g$ is $\lambda$-analytic in $\DD$ and
\begin{align*}
D_{z}\left(zg(z)\right)=\int_{\DD}D_{z}\left(z\widetilde{K}_{\lambda}(z,w)\right)D_{w}\left(wf(w)\right)(1-|w|^{2})d\sigma_{\lambda}(w),\qquad z\in\DD.
\end{align*}
In the meanwhile, from (\ref{Tzphi-2}), (\ref{Bergman-kernel-5}) and (\ref{Bergman-kernel-15}),
\begin{align*}
D_{z}\left(z\widetilde{K}_{\lambda}(z,w)\right)=K_{\lambda,1}(z,w),
\end{align*}
and hence
$$
D_{z}\left(zg(z)\right)=P_{\lambda,1}\left[D_{w}\left(wf(w)\right)\right](z),\qquad z\in\DD.
$$
Since, again by Lemma \ref{derivative-Bergman-a}, $D_{z}\left(zf(z)\right)\in L^{1}(\DD;(1-|z|^2)d\sigma_{\lambda})$, and from (\ref{derivative-analytic-1}), $D_{z}\left(zf(z)\right)$ is $\lambda$-analytic in $\DD$, it follows that $P_{\lambda,1}\left[D_{w}\left(wf(w)\right)\right](z)=D_{z}\left(zf(z)\right)$ by Proposition \ref{reproducing-Bergman-c}. Therefore
$$
D_{z}\left(zg(z)\right)=D_{z}\left(zf(z)\right),\qquad z\in\DD.
$$
In view of the $\lambda$-analyticity of $f$ and $g$, comparing the coefficients of $D_{z}\left(zg(z)\right)$ and $D_{z}\left(zf(z)\right)$ as in (\ref{derivative-analytic-1}) proves that $g=f$ in $\DD$.
\end{proof}

\subsection{Interpolation of the $\lambda$-Bergman spaces}

\begin{theorem}\label{Bergman-interpolation-a}
Suppose that $1\leq p_{0}<p_{1}<+\infty$ and
$$\frac{1}{p}=\frac{1-\theta}{p_{0}}+\frac{\theta}{p_{1}},\qquad \theta\in(0,1).$$
Then $[A^{p_{0}}_{\lambda}(\DD),A^{p_{1}}_{\lambda}(\DD)]_{\theta}=A^{p}_{\lambda}(\mathbb{D})$ with equivalent norms.
\end{theorem}

\begin{proof}
For $f\in A^{p}_{\lambda}(\mathbb{D})$, as usual we consider the family of functions
$$
f_{\zeta}(z)=\frac{f(z)}{|f(z)|}|f(z)|^{p\left(\frac{1-\zeta}{p_0}+\frac{\zeta}{p_1}\right)},\qquad z\in\DD,
$$
where $\zeta$ takes the value of the closed strip with its real part between $0$ and $1$. Further, let $F_{\zeta}(z)=(P_{\lambda,1}f)(z)$, that is,
\begin{align*}
F_{\zeta}(z)=
\int_{\DD}f_{\zeta}(w)K_{\lambda,1}(z,w)(1-|w|^2)d\sigma_{\lambda}(w),\qquad z\in\DD.
\end{align*}

By Proposition \ref{reproducing-Bergman-c}, for fixed $\zeta$, $F_{\zeta}(z)$ is $\lambda$-analytic in $\DD$, and $F_{\theta}(z)=f_{\theta}(z)=f(z)$; moreover by Theorem \ref{Bergman-projection-a},
\begin{align*}
\|F_{\zeta}\|^{p_0}_{A_{\lambda}^{p_0}}\lesssim\|f_{\zeta}\|^{p_0}_{L_{\lambda}^{p_0}(\DD)}\asymp\|f\|^p_{A_{\lambda}^{p}}
\end{align*}
for all $\zeta$ with ${\rm Re}\,\zeta=0$, and
\begin{align*}
\|F_{\zeta}\|^{p_1}_{A_{\lambda}^{p_1}}\lesssim\|f_{\zeta}\|^{p_1}_{L_{\lambda}^{p_1}(\DD)}\asymp\|f\|^p_{A_{\lambda}^{p}}
\end{align*}
for all $\zeta$ with ${\rm Re}\,\zeta=1$. It is easy to see that the mapping $\zeta\mapsto F_{\zeta}\in A_{\lambda}^{p_0}+A_{\lambda}^{p_1}$ is continuous when $0\le{\rm Re}\,\zeta\le1$ and (usual) analytic when $0<{\rm Re}\,\zeta<1$. Thus it has been proved that $f=F_{\theta}\in[A^{p_{0}}_{\lambda}(\DD),A^{p_{1}}_{\lambda}(\DD)]_{\theta}$ and $\|f\|^p_{[A^{p_{0}}_{\lambda},A^{p_{1}}_{\lambda}]_{\theta}}\lesssim\|f\|^p_{A_{\lambda}^{p}}$.

Conversely, if $f\in[A^{p_{0}}_{\lambda}(\DD),A^{p_{1}}_{\lambda}(\DD)]_{\theta}$, then $f$ is $\lambda$-analytic in $\DD$, and $f\in[L^{p_{0}}_{\lambda}(\DD),L^{p_{1}}_{\lambda}(\DD)]_{\theta}=L^{p}_{\lambda}(\DD)$. This shows that $f\in A^{p}_{\lambda}(\mathbb{D})$ and $\|f\|^p_{A_{\lambda}^{p}}\lesssim\|f\|^p_{[A^{p_{0}}_{\lambda},A^{p_{1}}_{\lambda}]_{\theta}}$. The proof of the theorem is completed.
\end{proof}

The above theorem for $\lambda=0$, i.e., for the usual Bergman spaces, was proved in \cite{MZ1}.

\vskip .3in

{\bf ACKNOWLEDGMENTS}

\vskip .1in

The work is supported by the National Natural Science Foundation of China (Grant No. 12071295).
The authors would like to thank the anonymous referees for their helpful comments and suggestions which have improved the original manuscript.

\end{document}